\documentclass[11pt]{article}
\usepackage[margin =1in]{geometry}
\usepackage{amssymb,amsthm,amsmath}
\usepackage{enumerate}
\usepackage{hyperref}
\usepackage{cite}
\usepackage[capitalize]{cleveref}
\usepackage{enumitem}
\usepackage{color}
\usepackage{cancel}
\usepackage{pgf}
\usepackage{svg}
\usepackage{pgfplots}
\usepackage{import}

\usepackage[utf8]{inputenc} 
\usepackage[T1]{fontenc}    
\usepackage{hyperref}       
\usepackage{url}            
\usepackage{booktabs}       
\usepackage{amsfonts}       
\usepackage{nicefrac}       
\usepackage{microtype}
\usepackage{multirow}
\usepackage{xfrac}
\usepackage{todonotes}
\usepackage{titlesec}
\usepackage{caption}
\usepackage{subcaption}
\usepackage{algorithm}
\usepackage{algorithmic}

\titleformat{\subsubsection}[runin]
{\normalfont\normalsize\bfseries}{\thesubsubsection}{1em}{}

\usepackage{graphicx}

\usepackage{appendix}
\numberwithin{equation}{section}

\newcommand{\ip}[1] {\langle #1 \rangle }

\newcommand{\inclu}[0] {\ar@{^{(}->}}

\newcommand{\R}{\mathbb{R}}

\newcommand{\RR}{\mathbb{R}}

\newcommand{\dom}{\text{dom }}
\newcommand{\epi}{\text{epi }}
\newcommand{\hypo}{\text{hypo }}
\newcommand{\interior}{\text{int }}
\newcommand{\argmin}{\operatornamewithlimits{argmin}}

\newcommand{\argmax}{\operatornamewithlimits{argmax}}

\newtheorem{thm}{Theorem}[section]
\newtheorem{theorem}{Theorem}[section]

\newtheorem{lemma}[thm]{Lemma}

\newtheorem{corollary}[thm]{Corollary}

\newtheorem{assumption}{Assumption}

\crefname{claim}{claim}{claims}
\Crefname{claim}{Claim}{Claims}
\crefname{lem}{lemma}{lemmas}
\Crefname{lem}{Lemma}{Lemmas}
\crefname{algorithm}{algorithm}{algorithms}
\Crefname{algorithm}{Algorithm}{Algorithms}

\theoremstyle{remark}

\newcommand{\varspace}{\mathcal{E}}
\newcommand{\defeq}{:=}
\newcommand{\gauge}[2]{\gamma_{#1, #2}}
\newcommand{\target}{\mathbb{R}_{++}}
\newcommand{\targetcl}{\overline{\mathbb{R}}_{++}}
\newcommand{\feasReg}{\mathcal{S}}
\newcommand{\ratioCons}{\rho}
\newcommand{\pMRM}{||\mbox{-}\mathtt{MRM}}

\newcommand{\set}[1]{\left\{ #1 \right\}}

\usepackage{mathtools}

\usepackage[algo2e, boxruled]{algorithm2e}

\usepackage[T1]{fontenc}
\usepackage{lmodern}
\usepackage[normalem]{ulem}

\begin{document}

	\title{Scalable Projection-Free Optimization Methods\\ via MultiRadial Duality Theory}

	 \author{Thabo Samakhoana\footnote{Johns Hopkins University, Department of Applied Mathematics and Statistics, \url{tsamakh1@jhu.edu}} \qquad Benjamin Grimmer\footnote{Johns Hopkins University, Department of Applied Mathematics and Statistics, \url{grimmer@jhu.edu}}}

	\date{}
	\maketitle

	\begin{abstract}
            Recent works have developed new projection-free first-order methods based on utilizing linesearches and normal vector computations to maintain feasibility. These oracles can be cheaper than orthogonal projection or linear optimization subroutines but have the drawback of requiring a known strictly feasible point to do these linesearches with respect to. In this work, we develop new theory and algorithms which can operate using these cheaper linesearches while only requiring knowledge of points strictly satisfying each constraint separately. Convergence theory for several resulting ``multiradial'' gradient methods is established. We also provide preliminary numerics showing performance is essentially independent of how one selects the reference points for synthetic quadratically constrained quadratic programs.
	\end{abstract}

    \section{Introduction}
Recently, several works~\cite{Friedlander2014,Aravkin2018,Renegar2016,Renegar2019,Grimmer2017,radial2,Zakaria2022,Lu2023,liu2023gauges} have proposed new projection-free first-order methods based on often cheap linesearches and normal vector computations with the feasible region. Such methods offer potential advantages in terms of their scalability over projected methods and conditional gradient/Frank-Wolfe-type methods as reliances on quadratic or linear optimization oracles as subroutines are avoided. Prior works based on such potentially cheaper linesearches have required knowledge of a ``good enough'' strictly feasible point to use as a reference. In the line of work by Grimmer~\cite{Grimmer2017,radial1,radial2}, these methods are called radial methods as linesearches based at the origin amount to searching along rays at each iteration. In this work, we circumvent the previous reliance on a known ``good enough'' strictly feasible point by developing a new family of ``MultiRadial Methods''. These methods instead rely on a collection of reference points, each only required to be feasible to one component of the problem's constraints.

Our primary interest is in the development of methods for maximization problems
\begin{align} \label{eq:main}
    p^* = \begin{cases} \text{max } f(x) \\
    \text{s.t. } x \in S_j \quad \text{for all } j = 1, \dots,  m
    \end{cases}
\end{align}
with concave objective function $f: \varspace \rightarrow  \mathbb{R} \cup \{-\infty\}$ and closed convex constraint sets $S_j\subseteq \varspace$ for some finite dimensional Euclidean space $\varspace$. No assumptions like Lipschitz continuity of $f$ are made.
We focus on the development of first-order methods where $f$ can be accessed through its function value, its (sup)gradients, and one-dimensional linesearches. Mirroring these three operations, we will only assume access to the sets $S_j$ via checking membership, its normal vectors, and one-dimensional linesearches. 

Alternative commonly utilized oracle models for the constraint sets $S_j$ can incur higher per-iteration computational costs. Orthogonal projections, commonly used in projected gradient methods, require quadratic optimization over each $S_j$ (or worse $\cap S_j$), which requires $S_j$ to be sufficiently simple this can be done in closed-form (or quickly approximated). Frank-Wolfe-type methods only require linear optimization at each iteration, which is often cheaper than projections but may still be prohibitive. Interior point-type methods are applicable when a self-concordant barrier function for each $S_j$ is available but require linear systems solves based at each iteration.

Lagrangian-type methods apply when the constraints take the functional form of $S_j = \{ x \mid g_j(x) \leq 0\}$, relying on first-order oracles for and the structure of each $g_j$. If each $g_j$ is convex but nonsmooth, a range of subgradient-type methods can be applied~\cite{Polyak1967,Metel2021}. If each $g_j$ is smooth, nearly optimal accelerated methods have been recently developed by Zhang and Lan~\cite{zhang2022solving}. An important distinction should be drawn between using first-order evaluations of functional constraints $g_j$ and our model of linesearches and normal vectors of $S_j$. Our oracle is independent of how one represents the set $S_j$. In contrast, the above referenced methods for functionally constrained problems may require careful preprocessing of constraints to perform well, as, for example, replacing $g_j(x)\leq 0$ with any positive rescaling $\lambda g_j(x) \leq 0$ will change their algorithm's trajectory.

Here we develop algorithms that access each constraint set $S_j$ by linesearches and normal vector computations. As linesearches, given some $e_j \in \interior S_j$ and $x \not\in S_j$, we assume one can find the unique point on the boundary of $S_j$ between $e_j$ and $x$. Even if this cannot be done in closed form, given a membership oracle for $S_j$, bisection or a similar rootfinding procedure could be used to reach a machine precision solution. Once a boundary point is produced, we assume a normal vector can be computed, mirroring the role of computing (sub)gradients of the objective. These two operations correspond to function evaluation and subgradient evaluation of the gauge of $S_j$ with respect to $e_j$, defined as
$$ \gamma_{S_j,e_j}(x) = \inf\set{v > 0 \mid e_j + \frac{x - e_j}{v} \in S_j} \  . $$
(A formal introduction and discussion of gauges is deferred to Section~\ref{subsec: gauges and radial reformulations}.)

These two oracles are often much cheaper (and hence lead to more scalable algorithms) than common alternatives. For example, consider any ellipsoidal constraint $S_j = \{ x \mid \|A_j x - b_j\|_2 \leq 1 \}$. Here our assumed linesearch and normal vector can be cheaply computed with closed forms: the one-dimensional linesearch is directly given by the quadratic formula and a normal vector follows from one matrix multiplication with $A_j^TA_j$. In contrast, linear optimization, projections, and interior point method steps on ellipsoids all require at least solving a linear system.

A family of projection-free algorithms only utilizing these cheaper oracles was first developed by Renegar~\cite{Renegar2016,Renegar2019}. We introduce these ideas following their more general development as ``radial algorithms'' of Grimmer~\cite{radial1,radial2}. These methods reformulate~\eqref{eq:main} as the equivalent radially dual problem\footnote{Note this radial dual is fundamentally different from the similarly named gauge dual of Freund~\cite{Freund1987} as knowledge of oracles for related conjugate functions and polar sets are avoided in the radial dual formulation.}
\begin{equation}\label{eq:classic-radial-dual}
    \min_{y} \max_j\{ f^{\Gamma,e}(y), \gamma_{S_j,e}(y)\}
\end{equation}
provided $f(e) > 0$ and  $e\in \mathrm{int} \cap S_j$. Here $\gamma_{S,e}$ is the gauge of $S_j$ with respect to $e$ and $f^{\Gamma,e}$ is a nonlinear transformation of $f$ (again see Section~\ref{subsec: gauges and radial reformulations} for formal definitions). This reformulation is quite amenable to the application of first-order methods since (i) it is unconstrained minimization, facilitating the use of projection-free methods, (ii) it only interacts with the constraints $S_j$ through their gauges, enabling the use of often cheaper oracles, and (iii) it is uniformly Lipschitz continuous, removing the need to assume such structure. However, the applicability of prior radial algorithms based on solving~\eqref{eq:classic-radial-dual} is limited by the required knowledge of a common strictly feasible point $e$. Indeed, the Lipschitz continuity of~\eqref{eq:classic-radial-dual} depends on how interior $e$ is to $\cap S_j$. So, a ``good'' reference point is very much needed for prior methods to be effective.

\paragraph{Our Contributions}
 The primary contribution of this work is generalizing the duality between the primal problem~\eqref{eq:main} and radially dual problem~\eqref{eq:classic-radial-dual} preserving the benefits (i)-(iii) above while avoiding any usage of a common point $e$. Instead, we consider the MultiRadially Dual problem
\begin{equation}\label{eq:multicenter-radial-dual}
    \min_{y} \max_j\{ f^{\Gamma,e_0}(y), \gamma_{S_j,e_j}(y)\}
\end{equation}
which only relies on separate points $e_0$ with $f(e_0)>0$ and $e_j \in \mathrm{int\ } S_j$ for each constraint. More generally, we develop theory relating~\eqref{eq:main} to any problem of the form $\min_{y} \max\{ f^{\Gamma,e_0}(y), \varphi(y)\}$ where $\varphi:\varspace \to \R$ is a convex function ``identifying'' the feasible region $\cap S_j = \{x \in \varspace \mid \varphi(x) \leq 1\}.$

\begin{enumerate}
    \item {\bf MultiRadial Duality Theory} We develop theory relating the optimal solutions of the primal problem~\eqref{eq:main} to those of~\eqref{eq:multicenter-radial-dual}. Our Theorems~\ref{thm: dualUpperBoundIfPrimalAboveOne} and~\ref{thm: primalLowerBoundIfDualBelowOne} provide direct, algorithmically useful bounds relating the primal and multiradial dual optimal values, controlled by a natural geometric condition number. In the special case where $p^*=1$, these bounds become tight and our Theorem~\ref{thm: primialDualEqualityAtOne} shows both problems have exactly the same solution sets.
    \item {\bf MultiRadial Methods} Based on this theory, we design and analyze new scalable, projection-free ``MultiRadial Methods''. For nonLipschitz nonsmooth convex optimization, our Corollary~\ref{cor: subgradRateWith||-MRM} guarantees a MultiRadial Subgradient Method converges at the optimal $O(1/\varepsilon^2)$ rate up to a log term, with each iteration computing at most one subgradient of $f$ or one normal vector of a constraint.
    When the objective and constraint sets are smooth, our Corollaries~\ref{cor: smoothRateWith||-MRM} and~\ref{cor: genGradRateWith||-MRM} show accelerated MultiRadial Smoothing and Generalized Gradient Methods converge at rates $O(1/\varepsilon)$ and $O(1/\sqrt{\varepsilon})$ up to a log term, where the latter relies on more expensive per-iteration computations with respect to $m$.
\end{enumerate}

\paragraph{Example - Convex Quadratically Constrained Quadratic Programming (QCQPs)}
Throughout this work, we periodically utilize quadratic optimization problems as a concrete, classic model to illustrate results. In particular, consider a convex QCQP
\begin{equation}\label{eq:qcqp}
    p^* = \begin{cases}
        \max & f_0(x) := r_0 - q_0^Tx - \frac{1}{2} x^TP_0x \\
        \mathrm{s.t.} & f_j(x) := r_j - q_j^Tx - \frac{1}{2} x^TP_jx \geq 0 \quad \forall j=1\dots m \ .
    \end{cases}
\end{equation}
for any positive semidefinite matrices $P_j$ and $p^*>0$.

For convex QCQPs, one natural selection for $e_0$ is the maximizer of the objective $f(x)$, given by solving $P_0 e + q_0 = 0$. Similarly, a natural selection of $e_j$ would be any solution of $P_j e +q_j = 0$. Our approach applies for any selection of $e_j$'s with $f_j(e_j)>0$. In Section~\ref{sec:numerics}, we numerically observe that the typical numerical performance of our MultiRadial Methods tends to be independent of the choice of centers $e_j.$ Consequently, it may suffice to cheaply approximate a solution of $P_j e +q_j = 0$.

Supposing each $P_j$ is positive definite, these selections correspond to $e_j= -P_j^{-1} q_j$ for $j=0,\dots, m$. Then the multiradial dual problem~\eqref{eq:multicenter-radial-dual} of~\eqref{eq:qcqp} takes the form
\begin{equation}\label{eq:perfect-dual-qcqp}
    \min_{y} \max_{j=1\dots m}\left\{ \frac{1 + \sqrt{1 + 2f_0(e_0)(y-e_0)^TP_0(y-e_0)}}{2f_0(e_0)} , \sqrt{\frac{(y-e_j)^TP_j(y-e_j)}{2f_j(e_j)}}\right\}\ .
\end{equation}
More generally, for any positive semidefinite $P_j$ and any selection of $e_j$ with $f_j(e_j) >0 $, the multiradial dual problem~\eqref{eq:multicenter-radial-dual} of~\eqref{eq:qcqp} remains describable in closed form as
\begin{align}
    \min_{y} \max_{j=1\dots m}\Bigg\{ &\frac{1-\nabla f_0(e_0)^T(y-e_0) + \sqrt{(1-\nabla f_0(e_0)^T(y-e_0))^2 + 2f_0(e_0) (y-e_0)^T P_0(y-e_0)}}{2f_0(e_0)} , \nonumber\\
    &\frac{-\nabla f_j(e_j)^T(y-e_j) + \sqrt{(\nabla f_j(e_j)^T(y-e_j))^2+2f_j(e_j)(y-e_j)^TP_j(y-e_j)}}{2f_j(e_j)}\Bigg\}\ .\label{eq:dual-qcqp}
\end{align}
In either case, each component of the objective and its gradient can be computed via one matrix-vector multiplication. In this sense, we claim the resulting multiradial first-order methods are ``scalable'' as many existing alternatives require at least a linear system solve each iteration. The development of method's only relying on matrix-vector multiplication has been a recent trend in linear programming~\cite{SeanLP,eclipse,deng2023new,LinLP} and quadratic programming~\cite{SeanQP}.

\paragraph{Outline} Section~\ref{section: prelims} introduces needed preliminaries. Our theory in Section~\ref{sec:theory} relates our unconstrained ``multiradial'' reformulations to the original problem and discusses immediate algorithmic consequences. Subsequently, in Section~\ref{sec:parallel-MRM}, we develop a parameter-free method based on approximately solving (rescalings of) these multiradial problems. Preliminary numerical results are presented in Section~\ref{sec:numerics} for QCQPs, validating our theory and highlighting one area where performance scales better than our theory predicts.
    \section{Preliminaries}\label{section: prelims}
Our notations follow those of the initial development of radial duality~\cite{radial1,radial2}, specialized to the convex settings considered here. We consider any finite-dimensional Euclidean space $\varspace$ with a norm $\|\cdot\|$ induced by an inner product $\ip{\cdot, \cdot}$. To apply previous radial theory, we restrict to consider objective functions with values in the (extended) positive reals, which we denote by $\targetcl = \target \cup \set{0, \infty}$. Here, $\target$ is the set of positive real numbers and $0,\infty$ should be interpreted as the limit points of $\target$, playing a similar role to $\pm\infty$ for the real numbers.

Throughout, we will primarily consider extended positive valued functions $f:\varspace \to \target \cup \set{0, \infty}$. We claim this restriction is minor: for any real-valued objective $\tilde f\colon \varspace \rightarrow \mathbb{R}\cup\{\pm\infty\}$ to be maximized, one can equivalently maximize the extended positive valued function $f(x) \defeq \max\{\tilde{f}(x) - \tilde{f}(x_0) + 1, 0\}$ when given any $x_0\in\varspace$ with $f(x_0)\in\mathbb{R}.$ 
For any extended real-valued function $f$, its effective domain, epigraph, and hypograph are
\begin{align*}
    \dom f & \defeq \set{x \in \varspace \mid f(x) \in \target} \\
    \epi f & \defeq \set{(x, u) \in \varspace \times \target  \mid f(x) \leq u} \\
    \hypo f & \defeq \set{(x, u) \in \varspace \times \target  \mid f(x) \geq u},
\end{align*}
respectively. We denote the closure of $\dom f$ by $S_0.$ A function $f:\varspace \to \targetcl$ is concave (convex) if $\hypo f$ $(\epi f)$ is convex. We say $f$ is upper (lower) semicontinuous at $x \in \varspace$ if $\limsup_{x' \to x} f(x') = f(x)$ ($\liminf_{x' \to x} f(x') = f(x)$) and say $f$ is globally upper (lower) semicontinuous if this holds for all $x\in\varspace$. We abbreviate upper (lower) semicontinuity as u.s.c.~(l.s.c.) at times.

\paragraph{Normals, Subdifferentials, and Smoothness.}
The inner product on $\varspace$ induces one on $\varspace \times \RR$ defined by $\ip{(x, u), (x', u')} \defeq \ip{x, x'} + u\cdot u'.$ We use the same notation for both inner products as it will be clear from context which is being used. 
We say that a vector $\xi$ is normal to a set $S$ at $x$ if $\ip{\xi, x' - x} \leq 0$ for all $x'\in S$. The set of all normal vectors to $S$ at $x$ is denoted by $N_S(x).$ A vector $\zeta \in \varspace$ is a subgradient of convex function $f$ at $x \in \varspace$ if $(\zeta, -1) \in N_{\epi f}((x, f(x)).$ The set of all subgradients of $f$ at $x$ is denoted by $\partial f(x)$ and referred to as the subdifferential of $f$ at $x.$ We say $\zeta \in \varspace$ is a supgradient of a concave function $f$ at $x \in \varspace$ if $(-\zeta, 1) \in N_{\hypo f}((x, f(x)).$ 
If $f$ is continuously differentiable, these differentials are exactly the singleton $\{\nabla f(x)\}$.

We say a function $f:\varspace \to \RR$ is $M$-Lipschitz continuous if $|f(x) - f(y)| \leq M\|x - y\|$ for all $x, y \in \varspace$ and a continuously differentiable function $f$ is $L$-smooth if its gradient is $L$-Lipschitz continuous on its domain. We say a set $S$ is $\beta$-smooth if any two unit length normal vectors $\xi_i\in N_S(x_i)$ for $i\in \{1,2\}$ satisfy $\|\xi_1 - \xi_2\| \leq \beta\|x_1-x_2\|$. A more detailed discussion on smooth sets is given in~\cite{liu2023gauges}.



\subsection{Minkowski Gauges and Radial Reformulations}\label{subsec: gauges and radial reformulations}
For any set $S\subseteq\varspace$, we define its gauge with respect to some $e\in S$ as 
\begin{equation}\label{def:gauge_with_center}
    \gauge{S}{e}(x) \defeq \inf\set{v > 0 \mid e + \frac{x - e}{v} \in S} \ .
\end{equation}
When $e=0$, this is the Minkowski gauge, denoted by $\gamma_{S}(y) = \inf\{v >0 \mid e/v\in S\}$. Otherwise, $\gauge{S}{e}$ can be viewed as a translation of the Minkowski gauge $\gamma_{S - e}.$ Note if $S$ is convex and $e \in \text{int } S,$ then $\gauge{S}{e}$ is convex, continuous and finite everywhere. 

This gauge of a set has a close relationship to the following indicator function. Namely, consider the nonstandard indicator function $\hat\iota_{S} : \varspace \to \set{0, \infty}$ defined as
\begin{equation}\label{def:indicfunc}
    \hat\iota_{S}(x) \defeq \begin{cases}+\infty & \text{if } x \in S \\ 0 & \text{otherwise .} \end{cases}
\end{equation}
To relate these functions, observe that the hypograph of this indicator has a bijection to the epigraph of the gauge of a closed convex $S$ with respect to any $e\in\interior S$ of
\begin{equation}\label{def:BigGammaTransform}
    \Gamma_e(x, u) \defeq \left(e + \frac{x - e}{u}, \frac{1}{u}\right) \ .
\end{equation}
Namely,
\begin{equation}\label{eq: indicatorEpi_GaugeHypoDuality}
    \hypo \hat\iota_S = \Gamma_e(\epi \gauge{S}{e}).
\end{equation}
This ``radial transformation'' $\Gamma_e$ was introduced in~\cite{radial1}, fixing $e=0$.

The epigraph-hypograph bijection~\eqref{eq: indicatorEpi_GaugeHypoDuality} motivates the following radial function transformation of a generic function $f:\varspace \to \targetcl$ with $e \in \varspace$ as\footnote{If the set on the right of \eqref{def: upperRadialTransform} is empty, we set $f^{\Gamma, e}(y)= 0$ rather than $-\infty$ to ensure the transformed function also maps into the extended positive reals.}
\begin{equation}\label{def: upperRadialTransform}
    f^{\Gamma, e}(y) \defeq \sup\set{v > 0 \mid (y, v) \in \Gamma_e(\epi f)} \ .
\end{equation}
Intuitively, one can view $f^{\Gamma,e}$ as the smallest function whose hypograph contains $\Gamma_e(\epi f).$ When $f =  \gauge{S}{e}$ for a closed convex set $S$ with $e \in \text{int } S,$ this radial transformation exactly turns gauges into indicator functions $\hat\iota_{S}^{\Gamma, e} = \gauge{S}{e}$.
Moreover, one can verify the reverse holds as well, $\hat\iota_{S} = \gauge{S}{e}^{\Gamma, e}$. So this transformation provides a bijection between indicator and gauge functions.

Expanding the definitions of $\Gamma_e$ and $\epi f$, one has $f^{\Gamma, e}(y) = \sup\set{v > 0 \mid vf(e + \frac{y - e}{v}) \leq 1}.$ When $e=0$, we ease notation, writing $f^\Gamma = f^{\Gamma, 0}$. From this, it becomes clear that $f^{\Gamma, e} = (f \circ \mathfrak{t}_{e})^{\Gamma}\circ \mathfrak{t}_{-e}$ where $\mathfrak{t}_{e}(y) = e + y$ denotes a translation by $e$, and so this radial transformation is just a translation of those proposed by Grimmer~\cite{radial1,radial2}. In the following, we summarize their results relating $f$ to $f^{\Gamma}$ and $(f^{\Gamma})^{\Gamma}$, emphasizing that $f^{\Gamma}$ can be replaced with $f^{\Gamma,e}$ for $e \in \varspace$ by the simple translation argument noted above. For more exposition, we refer the reader to the relevant parts of \cite{radial1} and \cite{radial2}.

The duality between indicators and gauges of convex sets carries over more generally to a wide range of (potentially nonconvex) functions. 
In particular, we say $f$ is upper radial with respect to $e$ if the translated perspective function $f^{p, e}(x, v) = v f(e + \frac{x - e}{v})$ is upper semicontinuous and nondecreasing in $v>0$ for all fixed $x\in \varspace$.
Theorem 1 of~\cite{radial1} establishes that this condition exactly characterizes when the radial function transformation is dual: For any $e\in \varspace$,
\begin{equation}
    (f^{\Gamma,e})^{\Gamma,e} = f \text{ if and only if } f \text{ is upper radial with respect to } e \ .
\end{equation}
The condition that $f^{p, e}(x, \cdot)$ is nondecreasing for all $x \in \varspace$ is equivalent to $\hypo f$ being star-convex with respect to $(e, 0),$ cf. \cite[Lemma 1]{radial1}. This duality between functions extends to give a duality between optimization problems as for any such objective: Proposition 24 of~\cite{radial1} ensures
\begin{equation}\label{eq:radial_duality_optimality}
    \left(\text{argmax } f\right) \times \set{\max f} = \Gamma_e\left(\left(\text{argmin } f^{\Gamma, e}\right) \times \set{f^{\Gamma, e}}\right) \ .
\end{equation}

\paragraph{Structural Properties of Gauges and Radial Reformulations.}
This work is primarily concerned with concave objective functions $f$ being maximized over convex sets $S_j$, for which the above star-convexity condition is easily verified. In this case, we can ensure a strengthened version of upper radiality holds: when $f$ is upper radial with respect to $e$ and $f^{p, e}(x, \cdot)$ is strictly increasing on $\dom f^{p, e}(x, \cdot) \defeq \set{v > 0 \mid f^{p,e}(x, v) \in\mathbb{R}_{++}}$ for every $x \in \varspace,$ we say $f$ strictly upper radial with respect to $e$. 
Then, it follows that all functions and sets considered here are well behaved as
\begin{align}
    f \text{ is concave and u.s.c.} &\implies f \text{ is strictly upper radial w.r.t.~any } e \in \interior \dom f \label{implication: f_convexImples_fisUpperRadial}\\
    S \text{ is convex and closed} &\implies \hat\iota_S \text{ is strictly upper radial w.r.t.~any } e \in \interior S\ .\label{implication: S_convexImples_iota_SisUpperRadial}
\end{align}
Given a bound on how interior $e$ is to the domain of $f$ (or to the constraint set $S$), we can further guarantee the radial transformation (or gauge) with respect to $e$ is well behaved, i.e., convex and uniformly Lipschitz continuous. Denote the interior radius of $S$ with respect to $e$ and diameter by 
\begin{align*}
    R_e(S) &\defeq \inf\set{\|x - e\| \mid x \notin S}\\
    D(S) &\defeq \sup\set{\|x - y\| \mid x,y \in S}
\end{align*}
Then~\cite[ Proposition 17]{radial1} and~\cite[Proposition 1, Lemma 1]{radial2} ensure the following
\begin{align}
    f \text{ is concave, u.s.c., and } R_e(S_0) >0 &\implies f^{\Gamma,e} \text{ is convex and } 1/R_{e}(S_0) \text{-Lipschitz}\ ,\label{implication:RLipschitzOfConvexDual}\\
    S \text{ is convex, closed, and } R_e(S)>0 &\implies \gauge{S}{e} \text{ is convex and } 1/R_{e}(S)\text{-Lipschitz}\ \label{implication: GaugesAreLipschitz}
\end{align}
where $S_0 = \mathrm{cl\ dom\ }f$.
Hence, provided ``good'' interior points to the domain of $f$ and each constraint are known, their transformations will be well-behaved and conditioned\footnote{In the nonconvex development of these radial transformations of~\cite{radial2}, these $R$ constants are generalized to measure how star-convex the given function's hypograph is.}.
Moreover, when $f$ is $L$-smooth or $S$ is $\beta$-smooth, this structure is preserved. Namely~\cite[Proposion 2]{radial2} ensures for twice continuously differentiable $f$ with bounded domain, $f^{\Gamma,e}$ is $O(L)$-smooth and~\cite[Theorem 3.2]{liu2023gauges} ensures for $\beta$-smooth, compact $S$, $\gauge{S}{e}^2$ is $O(\beta)$-smooth. Both big-O statements above suppress constants depending on the geometric radius and diameter quantities above.

Finally, we note three calculus/computational results of interest to our development. The family of upper and strictly upper radial functions is closed under many common operations, see \cite[Propositions 12 and 13]{radial1}:
If $f$ is (strictly) upper radial with respect to $e$, then so is $\lambda f$ for all $\lambda>0$ and
\begin{equation}
    (\lambda f)^{\Gamma, e} = \frac{1}{\lambda}f^{\Gamma, e} \circ \mathfrak{t}_e \circ \lambda \mathfrak{t}_{-e} \ .
\end{equation}
If $f_1,f_2$ are both (strictly) upper radial with respect to $e$, then so is $\min\{f_1,f_2\}$ and
\begin{equation}
    (\min\set{f_1, f_2})^{\Gamma, e} = \max\set{f_1^{\Gamma, e}, f_2^{\Gamma, e}}. \label{eq: radTransformDistributesOnMax}
\end{equation}
For any $f$ that is strictly upper radial with respect to some $e$, the subgradients of $f^{\Gamma, e}$ are easily computed from those of $f$ as~\cite[Proposition 19]{radial1} ensures
\begin{equation}
    \partial f^{\Gamma, e}(y) = \set{\frac{\zeta}{\ip{(\zeta, \delta),(x - e, u)}} \mid (\zeta, \delta) \in N_{\hypo f}((x, u)), \ip{(\zeta, \delta),(x - e, u)} > 0}
\end{equation}
where $(x, u) = \Gamma_e((y, f^{\Gamma, e}(y))).$

\subsection{First-Order Methods Minimizing Finite Maximums}
Instead of directly solving the primal problem~\eqref{eq:main}, our proposed MultiRadial Methods will solve (a sequence of) unconstrained convex minimization problems of the form~\eqref{eq:multicenter-radial-dual}. These reformulations will always be minimizing a finite maximum of convex functions:
\begin{equation} \label{eq: basicFiniteMaxMinimization}
    h_\star = \min_x \max\{h_0(x),\dots, h_m(x)\} \ .
\end{equation}
Let $h(x)=\max\{h_0(x),\dots, h_m(x)\}$ denote the whole objective being minimized. Depending on the structure of $f$ and $S_j$ in~\eqref{eq:main}, the multiradial dual will have components $h_j$ that are either Lipschitz or smooth. Below we review three well-known families of first-order methods capable of minimizing such objectives: first, the subgradient method for nonsmooth settings, and then accelerated smoothing and generalized gradient methods for smooth settings with large or small values of $m$, respectively.

Each first-order method $\mathtt{fom}$ considered maintains a sequence of iterates $y_i$ defined by two (simple) procedures for initializing/restarting itself and for taking one step. We denote the initialization process by $y_0 = \mathtt{fom.initialize}(x,\varepsilon,h)$, where $x\in\varspace$ is an initial solution, $\varepsilon>0$ is a target accuracy, and $h$ is the objective to minimize. For momentum methods, this procedure may involve initializing auxiliary variable sequences as well. We denote taking one step of $\mathtt{fom}$ by $y_{i+1}=\mathtt{fom.step}(y_i,\varepsilon,h)$, although auxiliary variable sequences may be updated as well. The considered methods all have convergence guarantees of the following form: If $\|y_0-y^*\| \leq D$ for some minimizer $y^*$ of $h$, then 
\begin{equation} \label{eq: subroutine-convergence}
    \text{Some } i \leq K_{\mathtt{fom}}(D,\varepsilon,h) \text{ has } h(y_i)-h(y^*)\leq \varepsilon \ .
\end{equation}

\paragraph{The Subgradient Method} The subgradient method, dubbed $\mathtt{subgrad}$, initializes simply with $y_0=x_0$ and iterates
\begin{equation}\label{eq:subgradient-method}
    y_{i+1} = y_i - \varepsilon g_i/\|g_i\|^2, \qquad g_i\in\partial h(y_i) \ .
\end{equation}
Note a subgradient of $h(x_k)$ can be computed as any subgradient of some $h_j(x_k)$ attaining the finite maximum. Provided each $h_j$ is convex and $M$-Lipschitz, which implies $h$ is convex and $M$-Lipschitz, the convergence of this method is well studied, having
$ K_{\mathtt{subgrad}}(D,\varepsilon,h) = M^2D^2/\varepsilon^2. $

\paragraph{The (Accelerated) Smoothing Method} Supposing instead that each $h_j$ is $L$-smooth and $M$-Lipschitz, one can utilize the smoothing techniques of~\cite{Nesterov2005,Beck2012}. Given a target accuracy $\varepsilon>0$, one can approximate $h$ by
$ h_{\theta}(y) = \theta \log\left(\sum_{j=0}^m \exp\left(\frac{h_j(y)}{\theta}\right)\right)$
for $\theta = \frac{\varepsilon}{2\log(m+1)}$. One can verify $h_\theta$ has $|h_\theta - h|\leq \varepsilon/2$ and is $L_\theta = L+\frac{M^2}{\theta}$-smooth. Then one can apply any accelerated gradient method to minimize $h_\theta$. For example, Nesterov's accelerated method initialized with $z_0=y_0, t_0=\frac{-1 + \sqrt{5}}{2}$ iterates
\begin{equation}\label{eq:smoothing-method}
    \begin{cases}
    y_{i+1} = z_i - \frac{1}{L_\theta}\nabla h_\theta(z_i) \\
    z_{i+1}=y_{i+1}+\beta_i(y_{i+1}-y_i) 
    \end{cases}
\end{equation}
where $t_{i+1}^2= (1-t_i)t_i^2$ and $\beta_i = t_i(1-t_i)/(t^2_i+t_{i+1})$. We denote this method by $\mathtt{smooth}.$ Noting any $\varepsilon/2$-minimizer of $h_\theta$ is an $\varepsilon$-minimizer of $h$, the accelerated convergence of $2\sqrt{L_\theta D^2/\varepsilon}$ in \cite[Theorem 2.2.3 ]{nesterov-textbook} gives a guarantee of the form~\eqref{eq: subroutine-convergence}
$$ K_{\mathtt{smooth}}(D,\varepsilon,h) = 2\sqrt{\frac{2LD^2}{\varepsilon}+\frac{4M^2D^2\log(m+1)}{\varepsilon^2}} \ . $$
In our numerics, we will instead use the Universal Fast Gradient Method (UFGM) of Nesterov~\cite{Nesterov2015UniversalGM}, which avoids requiring knowledge of $L_\theta$.

\paragraph{The (Accelerated) Generalized Gradient Method}
If, in addition to being $L$-smooth, the number of terms in the finite maximum $m$ is relatively small, one can utilize the generalized gradient method as outlined in~\cite{nesterov-textbook}. This method works by utilizing the generalized gradient mapping defined as
$$ \mathcal{G}(y,\alpha) = \frac{1}{\alpha}\argmin_{y'}\left\{\max_{j=0,\dots, m}\left\{h_j(y) + g_j^T(y'-y)\right\} + \frac{1}{2\alpha}\|y'-y\|^2\right\},\ \quad  g_j\in\partial h_j(y), $$
and then applying any accelerated method with $\mathcal{G}(y,\alpha)$ replacing the gradient, which we dub $\mathtt{genGrad}$. 
Computing $\mathcal{G}(y,\alpha)$ corresponds to solving a quadratic program of dimension $m+1$. This limits the applicability of such methods to settings where this can be efficiently calculated, primarily being useful when $m$ is small. Theorem 2.3.5 of~\cite{nesterov-textbook} ensures this method has a convergence guarantee of the form~\eqref{eq: subroutine-convergence} with $K_{\mathtt{genGrad}}(D,\varepsilon,h) =  2\sqrt{LD^2/\varepsilon}$.









    \section{MultiRadial Theory and Idealized Methods}\label{sec:theory}
We begin by developing our multiradial duality theory relating generic constrained maximization problems~\eqref{eq:main} to the unconstrained multiradially dual problem~\eqref{eq:multicenter-radial-dual}. Throughout, we will discuss immediate algorithmic implications by analyzing resulting simple multiradial algorithms. In the following section, we will propose and analyze a more practical parameter-free multiradial method.

First, we introduce some notations to describe the primal and (multi)radial dual objectives of~\eqref{eq:main} and~\eqref{eq:multicenter-radial-dual}. Let $\feasReg \defeq \bigcap_{j=1}^mS_j$ denote the primal feasible region and $\Psi: \varspace \to \targetcl$ denote the primal function
\begin{equation}\label{eq: PsiDefinition}
    \Psi(x) \defeq \min\set{f(x), \hat{\iota}_{\feasReg}(x)} \ .
\end{equation}
Maximizing $\Psi(x)$ is exactly the original primal problem~\eqref{eq:main} provided some $x\in \feasReg \cap \dom f$ exists, so $p^* \defeq \max_{x \in \feasReg}f(x) = \max_{x \in \varspace}\Psi(x).$ For any $e \in \interior (\feasReg \cap \dom f)$,
\begin{equation}\label{eq: common_center_dual}
    \Psi^{\Gamma, e} = \max\set{f^{\Gamma, e}, \gauge{\feasReg}{e}}
\end{equation}
by equation~\eqref{eq: radTransformDistributesOnMax} and the fact that $\gauge{\feasReg}{e} = \hat{\iota}_{\feasReg}^{\Gamma, e}.$
Thus the following duality relation holds
\begin{equation}\label{eq: common_center_solutions_duality}
    \left(\text{argmax } \Psi\right) \times \set{\max \Psi} = \Gamma_e\left(\left(\text{argmin } \Psi^{\Gamma, e}\right) \times \set{\min \Psi^{\Gamma, e}}\right) \ 
\end{equation}
by~\eqref{eq:radial_duality_optimality}. Requiring a point $e$ interior to every constraint is a notable limitation to the design of algorithms based on this relation. We address this by relaxing the dual objective function $\Psi^{\Gamma,e}.$ We instead consider the following dual function
\begin{equation}\label{eq:DualFunction}
    \Phi(y) = \max\{ f^{\Gamma,e_0}(y), \varphi_{\feasReg}(y)\}
\end{equation}
where $\varphi_{\feasReg}: \varspace \to \R$ is a l.s.c.~convex function satisfying $\interior \feasReg = \{x \in \varspace \mid \varphi_{\feasReg}(x) < 1\}.$ We call any such $\varphi_{\feasReg}$ a \emph{convex identifier} of $\feasReg.$ Based on equation~\eqref{eq: common_center_dual}, a natural choice for an identifier is $\varphi_{\feasReg} = \max\{\gamma_{S_1,e_1}, \dots, \gamma_{S_m,e_m}\}$ where $e_j \in \interior S_j$ for all $j.$ This particular $\varphi_{\feasReg}$ enables us to replace $e \in \interior (\feasReg \cap \dom f)$ with separate reference points for each functional component of the primal objective~\eqref{eq: PsiDefinition}. For this reason, we call $\Phi$ in~\eqref{eq:DualFunction} the multiradial dual function. We will at times refer to  $\max\{\gamma_{S_1,e_1}, \dots, \gamma_{S_m,e_m}\}$ as the canonical $\varphi_{\feasReg}$ and we encourage the reader to keep it as a concrete example of a convex identifier.


The primal problem and (multiradial) dual problem are then given by
\begin{equation}\label{eq: PrimalProblemWithPhi}
    p^* \defeq \max_{x \in \varspace} \Psi(x)
\end{equation}
\begin{equation}\label{eq: DualproblemWithPhi}
    d^* \defeq \min_{y \in \varspace}\Phi(y) \ .
\end{equation}
We will show that, under suitable assumptions, $\Phi$ is indeed an appropriate replacement to the radial dual $\Psi^{\Gamma, e}$ given by using a single reference point. A condition analogous to~\eqref{eq: common_center_solutions_duality} is derived in Theorem~\ref{thm: primialDualEqualityAtOne} in a restricted case, with general relationships being given in Theorems~\ref{thm: dualUpperBoundIfPrimalAboveOne} and~\ref{thm: primalLowerBoundIfDualBelowOne}. Note that with the canonical $\varphi_{\feasReg},$ the multiradial dual problem is an unconstrained, convex, uniformly Lipschitz minimization problem (and thus remains amenable to the direct application of many first-order methods).

Our theory relies on four assumptions, ensuring~\eqref{eq:main} is concave maximization with a maximizer and a Slater point, and that $\varphi_{\feasReg}$ and $f^{\Gamma,e_0}$ are well defined.
\begin{assumption}\label{assmpn: assumption1} 
$f$ is concave and u.s.c.~with bounded zero super-level set
$$
D_0 \defeq D(S_0) < \infty.
$$
\end{assumption}
\begin{assumption}\label{assmpn: assumption2}
    The constraint sets $S_1, \dots, S_m$ are convex and closed.
\end{assumption}
\begin{assumption}\label{assmpn: assumption3}
    A convex identifier $\varphi_{\feasReg}$ is known and a point $e_0 \in \interior S_0$ is known with
    \begin{align*}
        R_0 \defeq R_{e_0}(S_0) > 0.
    \end{align*}
\end{assumption}
\begin{assumption}\label{assmpn: assumption4}
    There exists $x^* \in \feasReg$ with $f(x^*) = p^* > 0$ and $x_{SL} \in \interior \feasReg \cap \dom f$ such that 
    $$
    \eta \defeq (1 - \gauge{S_0}{e_0}(x^*))(1 - \varphi_{\feasReg}(x_{SL})) > 0.
    $$
\end{assumption}

A few notes on these conditions. Firstly, under Assumptions~\ref{assmpn: assumption1} and~\ref{assmpn: assumption2}, Assumption~\ref{assmpn: assumption3} is satisfied if points $e_j \in \interior S_j$ are known for each $j = 0, 1, \dots, m.$ In this case, with $R \defeq \min\{R_{e_j}(S_j) \mid j = 0, 1, \dots, m\},$ $\Phi$ with the canonical $\varphi_{\feasReg}$ is $1/R$-Lipschitz continuous. Note the multiradial reformulation $\Phi$ can have a better Lipschitz constant than the radial dual~\eqref{eq:classic-radial-dual} relying on knowing a single $e \in \interior \cap_{j=0}^m S_m$ which is $1/R_e(\cap_{j=0}^m S_j)$-Lipschitz.
We leave the possibility of extending our optimality relationships between the primal and multiradial dual to nonconvex optimization to future works. Doing so would likely rely on replacing concavity assumptions by strictly upper radiality as done in~\cite{radial1}. However, such nonconvex problems are beyond the scope of the algorithms and analysis considered herein. Lastly, note that $x_{SL}$ will never be assumed to be known; it is only used in our analysis.

\subsection{Exact MultiRadial Dual Optimality Relationships}
These four assumptions suffice to show our primal and multiradially dual optimization problems are closely related. Our first result to this end is Theorem \ref{thm: primialDualEqualityAtOne}, which states that the two problems are equivalent when the optimal objective value is one, mirroring~\eqref{eq: common_center_solutions_duality}. This theorem is proved in Section~\ref{subsubsec: third_MRD_thm_proof}.
\begin{theorem}\label{thm: primialDualEqualityAtOne}
Under assumptions \ref{assmpn: assumption1} - \ref{assmpn: assumption4}, if $p^* = 1$ or $d^* = 1,$ then
$$
\argmax \Psi \times \set{p^*} = \argmin \Phi \times \set{d^*} \ .
$$
\end{theorem}
Problems with any $p^* > 0$ (not necessarily one) are still amenable to the application of this result by considering the rescaled primal function $\Psi_{\tau}$ and its multiradially dual function $\Phi_{\tau},$ given by
\begin{align}
    \Psi_{\tau}(x) &= \min\set{\tau f(x), \hat{\iota}_{\feasReg}(x)} \\
    \Phi_{\tau}(y) &= \max\{ (\tau f)^{\Gamma,e_0}(y), \varphi_\feasReg(y)\}
\end{align} 
for $\tau > 0.$ We let $p(\tau)$ and $d(\tau)$ respectively denote
\begin{align}
    p(\tau) \defeq \max_{x \in \varspace}\Psi_{\tau}(x) \\
    d(\tau) \defeq \min_{y \in \varspace}\Phi_{\tau}(y) \ .
\end{align}
By Theorem~\ref{thm: primialDualEqualityAtOne}, $p(\tau) = d(\tau) = 1$ whenever $\tau = \frac{1}{p^*}$ and these problems have the same set of solutions. Since $\argmax \Psi = \argmax \Psi_{\tau},$ the following duality relation holds.
\begin{equation} \label{eq:equal-optimizers}
    \argmax \Psi = \argmin \Phi_{1/ p^*} \ .
\end{equation}

For algorithmic purposes, requiring knowledge of $p^*$ is often prohibitive. As one example where such results are relevant, consider any minimization problem where strong duality holds. Then, minimizing the duality gap has a known optimal value, zero. To be concrete, consider a generic conic program over a closed convex cone $\mathcal{K}$ with dual cone $\mathcal{K}^*$ where the primal problem
minimizes $\langle c,x\rangle$ subject to $Ax=b$ and $x\in \mathcal{K}$ and the dual problem maximizes $\langle b,y\rangle$ subject to $c - A^*y \in \mathcal{K}^*$. Then one can formulate seeking optimal primal-dual solutions as the following problem with $p^*=1$
$$ 1=\begin{cases}
    \max & 1 + \langle b,y\rangle - \langle c,x\rangle\\
    \mathrm{s.t.} &Ax=b\\
    & x\in \mathcal{K}\\
    &c - A^*y \in \mathcal{K}^* \ . 
\end{cases} $$

\subsubsection{A Simple Method when the Optimal Value is Known} When $p^*$ is known and positive,~\eqref{eq:equal-optimizers} provides an alternative means to compute an approximate maximizer of the original problem. Given an initial point $x_0 \in \varspace$ and a given target accuracy $\varepsilon > 0,$ one could iterate
\begin{equation}\label{alg: theMRMwhenOptIsKnown}
    \begin{cases}
    y_0 = \mathtt{fom.initialize}(x_0, \varepsilon, \Phi_{1/p^*}) \\
    y_{i+1} = \mathtt{fom.step}(y_i, \varepsilon, \Phi_{1/p^*})
    \end{cases}.
\end{equation}
Guarantees on this scheme's convergence directly follow from the convergence rate $K_\mathtt{fom}(\cdot)$ of the given first-order method. The following theorem formalizes the primal objective gap and feasibility convergence of the above multiradial dual iterates $y_i$.
\begin{theorem}
    Under Assumptions~\ref{assmpn: assumption1}~-~\ref{assmpn: assumption4}, the points $z_i = e_0 + \frac{y_i - e_0}{\Phi_{1/p^*}(y_0)},$ where $y_i$ is the sequence \eqref{alg: theMRMwhenOptIsKnown}, have
    $$
    \frac{p^* - f(z_i)}{p^*} \leq \varepsilon \quad \text{and} \quad \inf_{x \in S_0 \cap \feasReg} \|z_i - x\| \leq \left[\frac{\varphi_\feasReg(e_0)D_0}{1 - \varphi_\feasReg(x_{SL})}\right]\varepsilon
    $$
    for some $i \leq K_{\mathtt{fom}}(\|x_0 - x^*\|, \varepsilon, \Phi_{1/p^*})$.
\end{theorem}
\begin{proof}
    Note some $i \leq K_{\mathtt{fom}}(\|x_0 - x^*\|, \varepsilon, \Phi_{1/p^*})$ must have $0 \leq \Phi_{1/p^*}(y_i) - 1 \leq \varepsilon.$ The claimed objective bound on the corresponding $z_i$ follows as
    \begin{align*}
        \frac{1}{p^*}f(z_i) &\geq \limsup_{v\searrow \Phi_{1/p^*}(y_i)} \frac{1}{p^*}f\left(e_0 + \frac{y_i - e_0}{v}\right)
         \geq \frac{1}{\Phi_{1/p^*}(y_i)}
         = 1 - \frac{\Phi_{1/p^*}(y_i) -1}{\Phi_{1/p^*}(y_i)} \geq 1 - \frac{\varepsilon}{\Phi_{1/p^*}(y_i)}
    \end{align*}
    where first inequality uses upper semicontinuity, the second uses the definition of $(f/p^*)^{\Gamma,e_0}$, and the third uses that $y_i$ is an $\varepsilon$-minimizer.
    The proof of our feasibility bound is deferred to Lemma~\ref{lem: dualHasBoundedLevelSets} showing $\inf_{x \in S_0 \cap \feasReg} \|z_i - x\| \leq \left[\frac{\varphi_\feasReg(e_0)D_0}{1 - \varphi_\feasReg(x_{SL})}\right]\frac{\Phi_{1/p^*}(y_i) - 1}{\Phi_{1/p^*}(y_i)}$.
\end{proof}

For example, consider the convex identifier $\varphi_S = \max\{\gamma_{S_j,e_j}\}$ as the maximum of the gauges of the constraint sets $S_j$ with respect to $e_j.$ Noting each gauge is $1/R_{e_j}(S_j)$-Lipschitz, the corresponding multiradial problem is $1/R$-Lipschitz where $R = \min_{j=0,\dots, m} R_{e_j}(S_j)$. Consequently, a multiradial subgradient method (that is, using the subgradient method~\eqref{eq:subgradient-method} in the multiradial method~\eqref{alg: theMRMwhenOptIsKnown}) requires at most
$$K_{\mathtt{subgrad}}(\|x_0 - x^*\|^2, \varepsilon, \Phi_{1/p^*}) = \frac{\|x_0 - x^*\|^2}{R^2\varepsilon^2} $$
iterations to produce some point with $\frac{p^* - f(z_i)}{p^*}\leq \varepsilon$ and $\inf_{x \in S_0 \cap \feasReg} \|z_i - x\| \leq \left[\frac{\varphi_\feasReg(e_0)D_0}{1 - \varphi_\feasReg(x_{SL})}\right]\varepsilon.$ Note this result is in line with prior radial subgradient method guarantees~\cite{Grimmer2017}, avoiding reliance on Lipschitz constant assumptions and instead only depending on ``geometric'' radius and diameter-type constants. Unlike these prior methods, a common $e\in\interior \cap S_j$ is not needed and as previously noted, the value of $R$ may be strictly larger.

\subsection{General MultiRadial Dual Optimality Relationships}
In the remainder of this section, we consider the relationship between $p^*$ and $d^*$ when $p^*$ is unknown, so simply rescaling the objective to have optimal value one beforehand is not doable. Our Theorems~\ref{thm: dualUpperBoundIfPrimalAboveOne} and \ref{thm: primalLowerBoundIfDualBelowOne} bound the absolute and relative distance of $p^*$ and $d^*$ from the value one in terms of each other. These theorems are proved in Section~\ref{subsec: MRD_theory_proofs}
\begin{theorem}\label{thm: dualUpperBoundIfPrimalAboveOne}
Under Assumptions~\ref{assmpn: assumption1}~-~\ref{assmpn: assumption4}, if $p^* - 1 \geq 0$ then 
$$
1 - d^* \geq \frac{R_0\eta}{R_0+D_0}\frac{p^* - 1}{p^*}.
$$
\end{theorem}
\begin{theorem}\label{thm: primalLowerBoundIfDualBelowOne}
Under Assumptions~\ref{assmpn: assumption1}~-~\ref{assmpn: assumption4}, if $1 - d^* \geq 0$ then 
$$
p^* - 1 \geq \frac{R_0}{D_0 + R_0}\frac{1 - d^*}{d^*}.
$$
In fact, if $y \in \mathcal{S}$ satisfies $f^{\Gamma, e_0}(y) \leq 1,$ then $f(y) - 1 \geq \frac{R_0}{D_0 + R_0}\frac{1 - f^{\Gamma, e_0}(y)}{f^{\Gamma, e_0}(y)}.$
\end{theorem}

These two theorems provide bounds on the relative distance from the primal/dual optimal value from one in terms of the dual/primal's optimal value's absolute gap from one. Such conversions between absolute and relative accuracy have occurred throughout prior works on radial methods, see Renegar~\cite{Renegar2016,Renegar2019}. For our multiradial theory, these relationships are primarily controlled by the natural geometric condition number based on the objective function's domain $R_0/(D_0+R_0)$.

Consider applying these bounds to a rescaled problem with objective function $\tau f$ for some $\tau \geq 1/p^*$. Recall this rescaled problem's maximum value is denoted by $p(\tau).$ In such rescaled settings, bounding $d^* \leq 1$, we denote the two coefficients above as
 \begin{equation}\label{def: conversionConstants}
    \ratioCons = \frac{R_0}{D_0 + R_0} \quad \text{and} \quad c_{\tau} = \frac{1}{p(\tau)}\frac{R_0}{D_0 + R_0}\eta.
 \end{equation}
 This notation helps illuminate the following relation implied by our theory
 \begin{equation}
     c_{\tau}[p(\tau) - 1] \leq 1 - d(\tau) \leq \frac{1}{\rho}[p(\tau) - 1] \quad \text{whenever } p(\tau) - 1 \geq 0 \text{ or } 1 - d(\tau) \geq 0.
 \end{equation}
The following corollaries of Theorems~\ref{thm: dualUpperBoundIfPrimalAboveOne}~and~\ref{thm: primalLowerBoundIfDualBelowOne} provide the basis for our algorithms.
\begin{corollary}\label{cor: fomWorksIfDeltaSmall}
     Let $ \tau \geq 1/ p^* > 0$ and $r \in [0, 1].$ Under Assumptions~\ref{assmpn: assumption1}~-~\ref{assmpn: assumption4}, any $y$ with $\Phi_{\tau}(y) - d(\tau)\leq (1 - r)[1 - d(\tau)]$ has
    $
    y \in \feasReg $ and $\tau f(y) - 1 \geq r\rho[1 - d(\tau)].
    $
\end{corollary}

\begin{proof}
    If $\Phi_{\tau}(y) - d(\tau) \leq (1 - r)[1 - d(\tau)],$ then $1 - \Phi_{\tau}(y) \geq r[1 - d(\tau)].$ Since $\tau \geq 1 / p^*$ implies $1 - d(\tau) \geq 0$ by Theorem~\ref{thm: dualUpperBoundIfPrimalAboveOne}, it follows that $y \in \feasReg$ as $\varphi_\feasReg(y) \leq \Phi_{\tau}(y) \leq 1.$ Moreover, by Theorem~\ref{thm: primalLowerBoundIfDualBelowOne},
        $\tau f(y) - 1 \geq \rho [1 - \Phi_{\tau}(y)] \geq r\rho[1 - d(\tau)]$.
\end{proof}

\begin{corollary}\label{cor: contractionWithRestartingCondition}
    Let $\tau_0 \geq 1/p^*,$ $\delta \geq 0,$ and $\mu \geq 1.$ Under Assumptions~\ref{assmpn: assumption1}~-~\ref{assmpn: assumption4}, if $\tau_1 \leq \frac{1}{1 + \delta}\tau_0$ and $1 - d(\tau) \leq \mu \delta$ then 
    $$
    p(\tau_1) - 1 \leq \frac{1}{1 + \delta}\left(1 - \frac{c_\tau}{\mu}\right)[p(\tau_0) - 1].
    $$
\end{corollary}
\begin{proof}
    Applying first that $\tau_1 \leq \frac{1}{1 + \delta}\tau_0,$ second that $1 - d(\tau_0) \leq \mu \delta,$ and third Theorem~\ref{thm: dualUpperBoundIfPrimalAboveOne}, one has
    \begin{align*}
        p(\tau_1) - 1 \leq \frac{1}{1 + \delta}p(\tau_0) - 1 = \frac{1}{1 + \delta}[p(\tau_0) - 1 - \delta] 
         &\leq \frac{1}{1 + \delta}\left[p(\tau_0) - 1 - \frac{c_{\tau_0}}{\mu}\frac{1 - d(\tau)}{c_{\tau_0}}\right] \\
         & \leq \frac{1}{1 + \delta}\left(1 - \frac{c_{\tau_0}}{\mu}\right)[p(\tau_0) - 1] \ . 
    \end{align*}
\end{proof}

\subsubsection{A Simple Method when Rescaled Problems can be Solved Exactly} To demonstrate how to benefit algorithmically from Theorems~\ref{thm: dualUpperBoundIfPrimalAboveOne} and~\ref{thm: primalLowerBoundIfDualBelowOne}, let $\tau_0 \geq 1 / p^*$ and consider the sequence 
 \begin{equation}\label{def: idealSequence}
     \begin{cases}
     y^{(k+1)} \in \argmin \Phi_{\tau_{k}}\\
    \tau_{k+1} = \frac{1}{f(y^{(k+1)})}.
     \end{cases}
 \end{equation}
With $r = 1,$ Corollary~\ref{cor: fomWorksIfDeltaSmall} implies $\tau_{k+1} \leq \frac{1}{1 + \rho[1 - d(\tau_k)]}\tau_k.$ Therefore, taking $\delta_k = \rho[1 - d(\tau_k)]$ and $\mu = 1/\rho,$ we have $\tau_{k+1} \leq \frac{1}{1 + \delta_k}\tau_k$ and $1 - d(\tau_k) \leq \mu \delta_k$ for all $k \geq 0.$ This implies $p(\tau_{k+1}) - 1 \leq \frac{1}{1 + \delta_k}(1 - \rho c_{\tau_k})[p(\tau_k) - 1]$ by Corollary~\ref{cor: contractionWithRestartingCondition}. Noting $c_{\tau_k} \geq c_{\tau_0}$ since $\tau_k$ is decreasing and $\delta_k\geq 0$ yields the following theorem.
\begin{theorem}\label{thm: linearConvergenceOfIdealSeq}
    Under Assumptions~\ref{assmpn: assumption1}~-~\ref{assmpn: assumption4}, if $\tau_0 p^* -1 > 0$, the ideal sequence~\eqref{def: idealSequence} has 
     $$     (1 - \ratioCons c_{\tau_0})^k[\tau_0p^* -1] \geq p(\tau_k) - 1. $$
    Hence for any $\varepsilon > 0$, all $k \geq \frac{1}{\ratioCons c_{\tau_0}}\log(p^*[\tau_0p^* - 1]/\varepsilon)$ have
     $$
     p^* - f(y^{(k)}) < \varepsilon \quad \text{and} \quad y^{(k)} \text{ feasible}.
     $$
 \end{theorem}

\subsubsection{A Simple Method when Rescaled Problems are Solved Inexactly}
The sequence~\eqref{def: idealSequence} will often not be practical to implement as it requires exact solutions to the multiradial dual problems $\min_{y\in \varspace}\Phi_{\tau}(y).$ However, the linear convergence in Theorem~\ref{thm: linearConvergenceOfIdealSeq} suggests that a good primal solution may be obtained by approximately solving a (relatively) small number of dual problems. This is the main motivation behind our general multiradial methods; we mimic the sequence~\eqref{def: idealSequence}, replacing exact solutions with approximate ones. 

Here we sketch a general family of methods of this form, with the drawback that determining when an approximate solution is good enough still requires unrealistic problem-dependent knowledge. 
We suppose an initial feasible point $y^{(0)} \in \dom f \cap \feasReg$ is given and set $\tau_0 = 1/f(y^{(0)})$.
To approximate the iteration~\eqref{def: idealSequence}, we apply a given $\mathtt{fom}$ to minimize $\Phi_{\tau_k}$ initialized at $y^{(k)}$, yielding iterates $y_i^{(k)}$. Once a sufficient accuracy is reached at some iteration $i$ of the subproblem optimization, we set $y^{(k+1)}=y^{(k)}_i$. One natural way to define sufficient accuracy is to require $\Phi_{\tau_k}(y^{(k)}_i) - d(\tau_k) \leq \delta_k.$ If $d(\tau_k) + \delta_k \leq 1,$ then such $y^{(k)}_i$ will be feasible and Theorem~\ref{thm: primalLowerBoundIfDualBelowOne} implies $1 + \rho\delta_k \leq \tau_k f(y^{(k)}_i).$ Motivated by this observation, we bypass the `dual' notion of accuracy and directly say $y^{(k)}_i$ is sufficiently accurate if (i) $1/f(y^{(k)}_i) \leq \frac{1}{1 + \delta_k}\tau_k$ and (ii) $y^{(k)}_i\in \cap_{j=1}^m S_j$. This process is formalized in Algorithm~\ref{alg: theMRM}.

\begin{algorithm}[ht]
\caption{The MultiRadial Method}\label{alg: theMRM}
\begin{algorithmic}[1]
\REQUIRE $(f, e_0),$ $x_0 \in \dom f \cap \feasReg$, a convex identifier $\varphi_\feasReg,$ $\set{\delta_k}_{k=0}^{\infty},$ a first-order method $\mathtt{fom}$
\STATE Set $\tau_{0} = 1/f(x_{0})$ and $y^{(0)}_0 = \mathtt{fom.initialize}(x_0, \delta_0, \Phi_{\tau_0})$ and $i=0$
\FOR{$k = 0, 1, 2, \dots,$}
    \REPEAT
        \STATE $y^{(k)}_{i+1} = \mathtt{fom.step}(y^{(k)}_{i},\delta_k,\Phi_{\tau_k}),\ i=i+1$ \hfill ($\mathtt{fom}$ takes one step)
    \UNTIL{$1/f(y^{(k)}_i) \leq \frac{1}{1 + \delta_k}\tau_k$ and $y^{(k)}_i$ is feasible}
    \STATE Set $\tau_{k+1} = 1/f(y^{(k)}_{i})$ \hfill (Restart $\mathtt{fom}$ once satisfied by $y^{(k)}_{i}$)
    \STATE Set $y^{(k+1)}_{0} = \mathtt{fom.initialize}(y^{(k)}_{i}, \delta_{k+1}, \Phi_{\tau_{k+1}})$ and $i=0$
\ENDFOR
\end{algorithmic}
\end{algorithm}

Selecting a sequence of stopping criteria $\delta_k$ for which Algorithm~\ref{alg: theMRM} has provably good performance guarantees is nontrivial. As $1 - d(\tau_k)$ decreases to $0$, $\delta_k$ must decrease similarly for the condition $1/f(y^{(k)}_i) \leq \frac{1}{1 + \delta_k}\tau_k$ to be attainable; 
below we show that taking $\delta_k = \rho\frac{1 - d(\tau_k)}{2}$ maintains the outer linear convergence rate of~\eqref{def: idealSequence}.


\begin{theorem}\label{thm: approxIdealSeqConvergence}
    Suppose Assumptions~\ref{assmpn: assumption1}~-~\ref{assmpn: assumption4} hold and let $\mathtt{fom}$ be given. Then, for all $\varepsilon > 0,$ Algorithm~\ref{alg: theMRM} with $\delta_k = \rho\frac{1 - d(\tau_k)}{2}$ has 

    $$
    p^* - f(y^{(k')}_0) \leq \varepsilon \quad \text{ and } \quad y^{(k')}_0 \text{ feasible}
    $$
    for some $k' \leq N \defeq \lceil\frac{2}{\rho c_{\tau_0}}\log(\frac{p^*[\tau_0 p^* - 1]}{\varepsilon})\rceil.$ The total number of $\mathtt{fom}$ steps needed to find such $y^{(k')}_0$ is at most $\sum_{k=1}^{N} K_{\mathtt{fom}}(D, \delta_{k}/\rho, \Phi_{\tau_{k}}).$
\end{theorem}

\begin{proof}
    Our proof is split into two parts. First, we bound the number of inner loop steps before the stopping criterion is met using Corollary~\ref{cor: fomWorksIfDeltaSmall}. Then, we bound the number of outer loop steps before an $\varepsilon$-minimizer by a similar contraction as seen for the exact method~\eqref{def: idealSequence} using Corollary~\ref{cor: contractionWithRestartingCondition}.

    By definition, the first-order method $\mathtt{fom}$ must have some iteration $i_k$ attain $\Phi_{\tau_k}(y^{(k)}_{i_k}) -d(\tau_k)\leq \delta_k/\rho$ with $i_k \leq  K_{\mathtt{fom}}(\|y^{(k)}_0 - y^{(k)}_*\|, \delta_{k}/\rho, \Phi_{\tau_{k}}) \leq K_{\mathtt{fom}}(D, \delta_{k}/\rho, \Phi_{\tau_{k}}).$
    Since $\delta_k/\rho = \frac{1 - d(\tau_k)}{2}$, Corollary~\ref{cor: fomWorksIfDeltaSmall} implies $y^{(k)}_{i_k}$ is feasible and $\tau_k f(y^{(k)}_{i_k}) - 1 \geq (1 - 1/2)\rho (1- d(\tau_k)) = \delta_k$ (or equivalently, $1/f(y^{(k)}_{i_k}) \leq \frac{\tau_k}{1+\delta_k}$). Therefore, $y^{(k)}_{i_k}$ would satisfy the stopping criteria for the inner loop of Algorithm~\ref{alg: theMRM} and so, the inner loop at iteration $k$ will always terminate within $K_{\mathtt{fom}}(D, \delta_{k}/\rho, \Phi_{\tau_{k}})$ steps.

    Notice that the inner loop stopping criterion ensures that $\tau_{k+1} \leq \frac{1}{1 + \delta_k}\tau_k,$ where $1 - d(\tau_k) \leq \frac{2}{\rho}\delta_k.$ Therefore, each outer loop contracts the (rescaled) objective gap towards one since
    \begin{align*}
        p(\tau_{k+1}) - 1 \leq \frac{1}{1 + \delta_k}\left(1 - \frac{\rho c_{\tau_k}}{2}\right)[p(\tau_k) - 1] &\leq \frac{1}{1 + \delta_k}\left(1 - \frac{\rho c_{\tau_0}}{2}\right)[p(\tau_k) - 1] \\
        &\leq \left(1 - \frac{\rho c_{\tau_0}}{2}\right)[p(\tau_k) - 1],
    \end{align*}
    where the first inequality used Corollary~\ref{cor: contractionWithRestartingCondition} and the second used that $c_{\tau_k}$ is increasing. As such, the primal gap converges linearly with $p(\tau_k) - 1 \leq (1 - \frac{\rho c_{\tau_0}}{2})^k[\tau_0 p^* - 1]$ and consequently, some $k' \leq \lceil\frac{2}{\rho c_{\tau_0}}\log(\frac{p^*[\tau_0 p^* - 1]}{\varepsilon})\rceil$ has a feasible $y^{(k')}_0$ with $p^* - f(y^{(k')}_0) \leq \varepsilon.$ Totalling the number of steps executed by $\mathtt{fom}$ to find each $y^{(k)}_0$ gives the claim.
\end{proof}

\subsection{Proofs for MultiRadial Duality Theory Optimality Relationships}\label{subsec: MRD_theory_proofs}
Below, we first prove Theorems \ref{thm: dualUpperBoundIfPrimalAboveOne} and~\ref{thm: primalLowerBoundIfDualBelowOne}, which bound the primal and dual difference from one in terms of each other. From these, Theorem~\ref{thm: primialDualEqualityAtOne} is almost immediate.

Our proofs use the following two facts repeatedly. Under Assumptions~\ref{assmpn: assumption1}~-~\ref{assmpn: assumption4},
\begin{equation}\label{ineq: concavityOnDomainClosure}
    f(\lambda x + (1 - \lambda)y) \geq \lambda f(x) + (1 - \lambda)f(y) \qquad \forall x,y\in S_0,
\end{equation}
\begin{equation}\label{ineq: fGammaBoundsDomainGauge}
    f(e_0 + \frac{y - e_0}{v}) \geq \frac{1}{v} \qquad \text{for all positive } v \geq f^{\Gamma, e_0}(y).
\end{equation}

\subsubsection{Proof of Theorem~\ref{thm: dualUpperBoundIfPrimalAboveOne}}\label{subsubsec: first_MRD_thm_proof}
    This result primarily follows from the following bound on the radial dual value at $x^*$
    \begin{equation}\label{ineq: fGammaAtOptUpperbound}
        f^{\Gamma, e_0}(x^*) \leq 1 - (1-\gauge{S_0}{e_0}(x^*))\frac{p^* - 1}{p^*}.
    \end{equation}
    We delay the proof of this inequality to first show it suffices to prove the theorem. Consider $x_{\lambda} = \lambda x_{SL} + (1 - \lambda)x^*$ with $\lambda = \frac{R_0}{R_0+D_0}(1 - \gauge{S_0}{e_0}(x^*))\frac{p^* - 1}{p^*}.$ This satisfies
    \begin{align*}
        f^{\Gamma, e_0}(x_\lambda) &\leq f^{\Gamma, e_0}(x^*) + \lambda\frac{D_0}{R_0} 
        \leq 1 - \frac{R_0\eta}{R_0+D_0}\frac{p^* - 1}{p^*}
    \end{align*}
    where the first inequality uses the $1/R_0$-Lipschitz continuity of $f^{\Gamma,e_0}$ and that $\|x_{SL} - x^*\|\leq D_0$, and the second uses~\eqref{ineq: fGammaAtOptUpperbound} and that $(1 - \gauge{S_0}{e_0}(x^*)) \geq \eta$. From the convexity of $\varphi_\feasReg$, it follows that
    $$
    \varphi_\feasReg(x_\lambda) \leq 1 - \frac{R_0\eta}{R_0+D_0}\frac{p^* - 1}{p^*}.
    $$
    Combined, these two bounds prove the result as $\Phi(x_{\lambda}) \leq 1 - \frac{R_0\eta}{R_0+D_0}\frac{p^* - 1}{p^*}.$

    Now we return to prove~\eqref{ineq: fGammaAtOptUpperbound}. Consider $\hat{v} \in (\gauge{S_0}{e_0}(x^*), 1)$ and let $\hat{x} = e_0 + \frac{x^* - e_0}{\hat{v}} \in S_0.$ For any $v \in [\hat{v}, 1]$, note that $e_0 + \frac{x^* - e_0}{v} = (1 - \alpha)\hat{x} + \alpha x^*$ where $\alpha = 1 - \frac{\hat{v}(1 - v)}{v(1 - \hat{v})} \in [0, 1].$ Since $x^*,\hat x \in S_0$, it follows that
    \begin{align*}
        f\left(e_0 + \frac{x^* - e_0}{v}\right) \geq \alpha p^* + (1 - \alpha)f(\hat{x}) \geq \alpha p^* = \frac{v - \hat{v}}{v(1 - \hat{v})} p^*
    \end{align*}
    where the first inequality uses \eqref{ineq: concavityOnDomainClosure} and the second uses $f(\hat{x}) \geq 0.$ 
    In particular, for $v\in [\hat{v} + [1 - \hat{v}]/p^*, 1],$ this ensures $vf\left(e_0 + \frac{x^* - e_0}{v}\right) \geq 1$. Hence $f^{\Gamma, e_0}(x^*) \leq 1 - (1 - \hat{v})(p^* - 1)/p^*$ since $f$ is strictly upper radial. Taking the limit as $\hat{v} \rightarrow \gauge{S_0}{e_0}(x^*)$ gives the claim.

\subsubsection{Proof of Theorem~\ref{thm: primalLowerBoundIfDualBelowOne}}\label{subsubsec: second_MRD_thm_proof}
We separate the proof into two lemmas which are of interest in their own right. The first lemma shows that the multiradial dual function has bounded level sets. Since this function is convex and bounded below, the lemma guarantees that the function has global minimizers. The second lemma establishes the inequality $f(y) - 1 \geq \frac{R_0}{R_0 + D_0}\frac{1 - f^{\Gamma, e_0}(y)}{f^{\Gamma, e_0}(y)}.$ Combining these two lemmas immediately completes the proof.
\begin{lemma}\label{lem: dualHasBoundedLevelSets}
    Under Assumptions~\ref{assmpn: assumption1}~-~\ref{assmpn: assumption4}, if $\Phi(y) \leq 1 + \varepsilon$ for $\varepsilon \geq 0,$ then $y_{\varepsilon} = e_0 + \frac{y - e_0}{1 + \varepsilon}$ has $\inf\limits_{x \in S_0 \cap \feasReg}\|y_{\varepsilon} - x\| \leq \frac{\varepsilon}{1 + \varepsilon}\frac{\varphi_\feasReg(e_0)}{1 - \varphi_\feasReg(x_{SL})}D_0.$ In particular, $\|y - e_0\| \leq D_0 + \varepsilon\left[1 + \frac{\varphi_\feasReg(e_0)}{1 - \varphi_\feasReg(x_{SL})}\right]D_0.$
\end{lemma}
\begin{proof}[Proof of Lemma~\ref{lem: dualHasBoundedLevelSets}]
    Consider the point $x_\lambda = (1 - \lambda)x_{SL} + \lambda y_{\varepsilon}$ with $\lambda = \frac{(1 - \varphi_\feasReg(x_{SL}))(1 + \varepsilon)}{(1 - \varphi_\feasReg(x_{SL}))(1 + \varepsilon) + \varepsilon\varphi_\feasReg(e_0)} \in [0,1]$. First, observe $x_\lambda\in S_0$ follows from convexity of $S_0$ since $x_{SL}\in S_0$ by definition and $y_\varepsilon\in S_0$ by \eqref{ineq: fGammaBoundsDomainGauge} noting $f^{\Gamma,e_0}(y)\leq \Phi(y)\leq 1+\varepsilon$. Next, observe $x_\lambda\in \feasReg$ by our choice of $\lambda$ as
    \begin{align*}
        \varphi_\feasReg(x_\lambda) &\leq \varphi_\feasReg(x_{SL}) + \lambda[\varphi_\feasReg(y_{\varepsilon}) - \varphi_\feasReg(x_{SL})] \\
        &\leq \varphi_\feasReg(x_{SL}) + \lambda\left[1 + \frac{\varepsilon}{1 + \varepsilon}\varphi_\feasReg(e_0) - \varphi_\feasReg(x_{SL})\right] = 1
    \end{align*}
    where the first inequality uses convexity of $\varphi_\feasReg$ at $x_\lambda$, and the second uses convexity of $\varphi_\feasReg$ at $y_\varepsilon$ and that $\varphi_\feasReg(y) \leq \Phi(y)\leq 1+\epsilon$. Together since $x_\lambda\in S_0\cap\feasReg$, we conclude
    $$\inf\limits_{x \in S_0 \cap \feasReg}\|y_{\varepsilon} - x\| \leq \|y_\varepsilon-x_\lambda\| = \frac{\varepsilon}{1 + \varepsilon}\frac{\varphi_\feasReg(e_0)}{1 - \varphi_\feasReg(x_{SL})} \|x_\lambda - x_{SL}\|.$$
    Bounding $\|x_\lambda - x_{SL}\|\leq D_0$ gives the lemma's first claim. The second claim follows similarly, noting $\|x_\lambda-e_0\|\leq D_0$ as well and applying the triangle inequality
    \begin{align*}
        \|y - e_0\| = (1 + \varepsilon) \|y_{\varepsilon} - e_0\| & \leq (1 + \varepsilon)(\|y_{\varepsilon} - x_{\lambda}\| + \|x_{\lambda} - e_0\|) \\
        &\leq D_0 + \varepsilon\left[1 + \frac{\varphi_\feasReg(e_0)}{1 - \varphi_\feasReg(x_{SL})}\right]D_0.
    \end{align*}
\end{proof}
\begin{lemma}\label{lem: lowerBoundOnfByDualUpperBound}
    Suppose $\hypo f$ is convex and $f$ is globally upper semi-continuous. Then, for any $e_0 \in \dom f,$  $0 \leq f^{\Gamma, e_0}(y) \leq 1$ implies $f(y) - 1 \geq \frac{(1 - f^{\Gamma, e_0}(y))R_0}{\|y - e_0\| + f^{\Gamma, e_0}(y) R_0}.$ \footnote{If $R_0 = \infty$ this should be taken as $f(y) \geq 1 + \frac{1 - f^{\Gamma, e_0}(y)}{f^{\Gamma, e_0}(y)}$} In addition, if $D_0 \defeq D(S_0) < \infty,$ then $f(y) \geq 1 + \frac{R_0}{D_0 + R_0} \frac{1 - f^{\Gamma, e_0}(y)}{f^{\Gamma, e_0}(y)}.$
\end{lemma}

\begin{proof}[Proof of Lemma~\ref{lem: lowerBoundOnfByDualUpperBound}]
Fix $y$ with $0 \leq 1 - f^{\Gamma, e_0}(y) \leq 1.$ If $R_0 = \infty$ then $f$ is strictly positive and concave on $\varspace$ which is possible only if $f$ is constant. If $f$ is constant, then so is $f^{\Gamma, e_0},$ hence $f^{\Gamma, e_0}(y) = f^{\Gamma, e_0}(e_0) = \frac{1}{f(e_0)}.$ Therefore, if $R_0 = \infty$ or $y = e_0,$ then $f(y) = 1 / f^{\Gamma, e_0}(y) = 1 + \frac{1 - f^{\Gamma, e_0}(y)}{f^{\Gamma, e_0}(y)}.$

Suppose for the rest of the proof that $y \neq e_0$ and $R_0 < \infty.$ Let $z = e_0 - R_0 \frac{y - e_0}{\|y-e_0\|} \in S_0.$ For any positive $v \geq f^{\Gamma, e_0}(y),$ let $w = e_0 + \frac{y - e_0}{v}$ and $\lambda = v\left(1 +  \frac{R_0(1 - v)}{\|y - e_0\| + R_0 v}\right)$. Then noting $y=\lambda w + (1-\lambda)z$, it follows that
\begin{align*}
    f(y) &= f(\lambda w + (1-\lambda)z)
    \geq \lambda f(w) + (1-\lambda)f(z)
    \geq \frac{\lambda}{v} = 1 +  \frac{R_0(1 - v)}{\|y - e_0\| + R_0 v}
\end{align*}
where the first inequality uses \eqref{ineq: concavityOnDomainClosure}, and the second uses \eqref{ineq: fGammaBoundsDomainGauge} and that $f(z)\geq 0$. Taking the limit as $v \to f^{\Gamma, e_0}(y)$ gives $f(y) \geq \frac{R_0(1 - f^{\Gamma, e_0}(y))}{\|y - e_0\| + R_0 f^{\Gamma, e_0}(y)}.$ The second part of the lemma follows from
\begin{align*}
    f^{\Gamma, e_0}(y) \geq \gauge{S_0}{e_0}(y) \geq \frac{1}{D_0}\|y - e_0\|,
\end{align*}
where the first inequality follows from \eqref{ineq: fGammaBoundsDomainGauge} and the second follows because if $S_0  \subset B(e_0, D_0) \defeq \set{x \in \varspace \mid \|x - e_0\| \leq D_0},$ then $\frac{1}{D_0}\|\cdot - e_0\| = \gauge{B(e_0, D_0)}{e_0} \leq \gauge{S_0}{e_0}.$ 
\end{proof}

\subsubsection{Proof of Theorem~\ref{thm: primialDualEqualityAtOne}}\label{subsubsec: third_MRD_thm_proof}
    Theorems~\ref{thm: dualUpperBoundIfPrimalAboveOne} and~\ref{thm: primalLowerBoundIfDualBelowOne} imply that $p^* = 1$ if and only if $d^* = 1.$ It remains to show that the two problems have the same solutions. To that end, suppose $x^* \in \mathcal{S}$ and $f(x^*) = p^* = 1.$ Since $f$ is strictly upper radial, $f(x^*) = 1$ implies $f^{\Gamma, e_0}(x^*) = 1.$ From $x^* \in \mathcal{S},$ we get $\varphi_\feasReg(x^*) \leq 1,$ hence $\Phi(x^*) = f^{\Gamma, e_0}(x^*) = 1 = d^*.$
    On the other hand, suppose $\Phi(y^*) = d^* = 1.$ Then, $\varphi_\feasReg(y^*) \leq 1$ hence $y^* \in \mathcal{S}$ and $f(y^*) \leq p^* = 1.$ We also have $f^{\Gamma,e_0}(y^*) \leq 1$ which, by the second part of Theorem~\ref{thm: primalLowerBoundIfDualBelowOne}, implies $f(y^*) \geq 1 + \frac{R_0}{R_0+D_0}(1 - f^{\Gamma,e_0}(y^*)) \geq 1.$ Therefore, $f(y^*) = 1 = p^*,$ and the proof is complete.

    \section{A Parameter-Free, Optimal, Parallel MultiRadial Method} \label{sec:parallel-MRM}

The previously discussed multiradial method in Algorithm~\ref{alg: theMRM} required unrealistic knowledge to compute the needed $\delta_k.$ In this section, we present a parameter-free adaption of this method, using the parallel restarting ideas of~\cite{Renegar2022}, which we call the Parallel MultiRadial Method ($||\mbox{-}\mathtt{MRM}$).

Conceptually, the $||\mbox{-}\mathtt{MRM}$ can be thought of as consisting of $N$ parallel, but not independent, instances of Algorithm~\ref{alg: theMRM}. In each $l$-th instance, Algorithm~\ref{alg: theMRM} is run with a constant accuracy sequence $\set{\delta^{(l)}_k}_{k=0}^{\infty} = \set{\delta^{(l)}}.$ All instances start with the same initial data $(f, e_0),$ convex identifier $\varphi_\feasReg$, feasible $x_0$, and $\mathtt{fom}.$ With a slight abuse of notation, we denote each instance by $\mathtt{fom}^{(l)}.$ The crucial part of $\pMRM$ is that the instances cooperate by sharing their feasible iterates with one another. In particular, one instance, say $\mathtt{fom}^{(1)},$ can use an iterate of another, say $\mathtt{fom}^{(2)},$ to make the update in step 6 of Algorithm~\ref{alg: theMRM}, if such an iterate is sufficiently accurate for $\mathtt{fom}^{(1)}.$ In the end, the best feasible iterate among all is returned as the solution. The number of instances $N$ and the respective target accuracy $\delta^{(l)}$ for each instance can be treated as inputs to the method. A concrete implementation of the $\pMRM$ is given in Algorithm~\ref{alg: ||-MRM}. For simplicity, Algorithm~\ref{alg: ||-MRM} takes $b \geq 2$ and $N$ as inputs and automatically sets $\delta^{(l)} = b^{-l}.$ Motivated by Theorem~\ref{thm: approxIdealSeqConvergence}, we find that setting $N = O(\log_b(1/\varepsilon))$ is sufficient to reach $\varepsilon$-accuracy.


More formally, Algorithm~\ref{alg: ||-MRM} produces iterates $y^{(l)}_i,$ $i = 0, 1, \dots,$ $l = 1, \dots, N.$ For each $l,$ the next iterate $y^{(l)}_{i+1}$ is produced from the previous one by a single step of $\mathtt{fom}^{(l)} = \mathtt{fom},$ i.e., $y^{(l)}_{i+1} = \mathtt{fom}^{(l)}\mathtt{.step}(y^{(l)}_{i},\delta^{(l)},\Phi_{\tau^{(l)}_i}).$ These steps can be done in parallel or, as described in Algorithm~\ref{alg: ||-MRM}, sequentially. Then the best iterate among all past and present feasible iterates, denoted $y^{best}_{i+1},$ is computed. With $\tau^{best}_{i+1} = 1/f(y^{best}_{i+1}),$ each instance $\mathtt{fom}^{(l)}$ for which $y^{best}_{i+1}$ meets their restarting criteria (e.g., $\tau^{best}_{i+1} \leq \frac{1}{1+\delta^{(l)}}\tau^{(l)}_i$), will set their $\tau^{(l)}_{i+1}$ as $\tau^{best}_{i+1}$ and reinitialize $\mathtt{fom}^{(l)}$ at $y^{best}_{i+1}$. We refer to this event as $\mathtt{fom}^{(l)}$ restarting.



\begin{algorithm}[h]
\caption{The Parallel MultiRadial Method ($||\mbox{-}\mathtt{MRM}$)}\label{alg: ||-MRM}
\begin{algorithmic}[1]
\REQUIRE $(f, e_0),$ $x_0 \in \dom f \cap \feasReg$, a convex identifier $\varphi_\feasReg,$ $b \geq 2,$ $N>0$, a first-order method $\mathtt{fom}$
\STATE Set $\delta^{(l)} = b^{-l}$ for each $l=1, \dots, N$
\STATE Set each $\tau_0^{(l)} = 1/f(x_0)$ and  $y_0^{(l)} = \mathtt{fom}^{(l)}\mathtt{.initialize}(x_0, \delta^{(l)}_0, \Phi_{\tau^{(l)}_0})$
\FOR{$i = 0, 1, 2, \dots,$}
    \FOR{$l= 1\dots N$}
        \STATE $\tau_{i+1}^{(l)} = \tau_{i}^{(l)}$ \hfill (Each $\mathtt{fom}^{(l)}$ takes one step)
        \STATE $y^{(l)}_{i+1} = \mathtt{fom}^{(l)}\mathtt{.step}(y^{(l)}_{i},\delta^{(l)},\Phi_{\tau^{(l)}_i})$ 
    \ENDFOR
    \STATE $y^{best}_{i+1} = \argmin\{1/f(y^{(l)}_{i'}) \mid y^{(l)}_{i'} \text{ is feasible and } i'\leq i+1\}$ \hfill (Find the best iterate seen thus far)
    \STATE $\tau^{best}_{i+1} = 1/f(y^{best}_{i+1})$
    \FOR{each $l=1\dots N$ with $\tau^{best}_{i+1} \leq \frac{1}{1+\delta^{(l)}}\tau^{(l)}_i$} 
        \STATE Set $\tau_{i+1}^{(l)} = \tau^{best}_{i+1}$ \hfill (Restart $\mathtt{fom}^{(l)}$ if satisfied by $y^{best}_{i+1}$)
        \STATE Set $y_{i+1}^{(l)} = \mathtt{fom}^{(l)}\mathtt{.initialize}(y^{best}_{i+1}, \delta^{(l)}, \Phi_{\tau^{(l)}_{i+1}})$ 
    \ENDFOR
\ENDFOR
\end{algorithmic}
\end{algorithm}

\subsection{Convergence Guarantees and Theory}
For ease of exposition, we define a few additional quantities not explicitly used in $||\mbox{-}\mathtt{MRM}$:
For each instance of the first-order method $l$, we let $i^{(l)}_0<i^{(l)}_1<i^{(l)}_2<\dots<i^{(l)}_{K_l}$ be the sequence of iterations where a restart occurred (i.e., $\tau^{best}_{i^{(l)}_k+1} \leq \frac{1}{1+\delta^{(l)}}\tau^{(l)}_{i^{(l)}_k}$). Note there must only be a finite number of such events. Each restart has
$$\tau^{(l)}_{i^{(l)}_k +1} \leq \frac{1}{1+\delta^{(l)}}\tau^{(l)}_{i^{(l)}_k} \ . $$
Inductively applying this and noting $1/f(y^{best}_{i+1}) \geq 1/p^*$, the total number of restarts by first-order method instance $l$ is at most $K_l \leq \log(p^*/f(x_0))/\log(1 + \delta^{(l)}) $.
For each iteration $i$, we say the {\it critical first-order method instance} $l_i$ is the one with $\frac{b\delta^{(l_i)}}{\rho} \leq 1 - d(\tau^{best}_{i}) < \frac{b^2\delta^{(l_i)}}{\rho}.$

The following three quantities are useful in formalizing our convergence rate guarantees for Algorithm~\ref{alg: ||-MRM}. They correspond to bounds on how long it takes for the instance $l$ to be guaranteed to reach a $\delta^{(l)}$-minimizer of any of its subproblem, the first parallel instance that could be critical, and given some $\varepsilon>0$, the last parallel instance that can be critical before an $\varepsilon$-minimizer is found.
\begin{align}
    K^{(l)}_{\mathtt{fom}} &\defeq \max_{i \geq 0}K_{\mathtt{fom}}(D_0, \delta^{(l)}/\rho, \Phi_{\tau^{(l)}_i}) \\
    l_0 & \defeq \min\set{l = 1, 2, \dots \mid \frac{b \delta^{(l)}}{\rho} < 1} = \lfloor\log_b(b^2/\rho)\rfloor \\
    \tilde{N}(\varepsilon) & \defeq \max\set{l = 1, 2, \dots \mid \frac{b^2\delta^{(l)}}{\rho} \geq \frac{c_{\tau_0}\varepsilon}{p^*}} = \left\lfloor\log_b\left(\frac{b^2p^*}{c_{\tau_0}\rho \varepsilon}\right)\right\rfloor \ .
\end{align}

Based on these, we have the following convergence guarantee (proof deferred to Section~\ref{sec:proofs}), establishing that $\pMRM$ needs a logarithmic number $O(\tilde N(\varepsilon))$ of multiradial dual problem solves, each requiring a number of steps controlled by the chosen first-order method, $K^{(l)}_{\mathtt{fom}}$.
\begin{theorem}\label{thm: main||-MRMtheorem}
    Suppose Assumptions~\ref{assmpn: assumption1}~-~\ref{assmpn: assumption4} hold and let $\mathtt{fom}$ be given. Then for all $\varepsilon > 0,$ Algorithm~\ref{alg: ||-MRM} with $\mathtt{fom},$ $b \geq 2,$ and $N \geq \log_b(\frac{bp^*}{c_{\tau_0}\rho \varepsilon})$ has 
    $$
    p^* - f(y^{best}_i) \leq \varepsilon \quad \text{and} \quad y^{best}_i \text{ feasible}
    $$
    provided that\footnote{We use the convention that $\sum_{l = a}^{b} K^{(l)}_{\mathtt{fom}} = 0$ if $b < a.$}  
    $$
    i \geq \frac{\log(\tau_0 p^*)}{\log(1 + \frac{\rho}{b^2})} K^{(l_0)}_{\mathtt{fom}} + \frac{5b^2(1 + \frac{\rho}{b^2})}{4\rho c_{\tau_0}}\sum_{l = l_0 + 1}^{\tilde N(\varepsilon)} K^{(l)}_{\mathtt{fom}} \ .
    $$
\end{theorem}

To illustrate the reach of this theorem, we present corollaries for three pairs of first-order methods (previously introduced in Section~\ref{section: prelims}) and appropriately chosen convex identifiers. Detailed proofs of these corollaries are deferred to Appendix~\ref{app:proofs}. Suppose points $e_j\in\mathrm{int\ } S_j$ are known.
First, noting that the gauges $\gamma_{S_j,e_j}$ are uniformly Lipschitz, one may reasonably consider
\begin{align}
    \mathtt{fom} = \mathtt{subgrad} \qquad \text{and} \qquad \varphi_\feasReg = \max\{\gamma_{S_1,e_1},\dots,\gamma_{S_m,e_m}\} \ . \label{eq:subgrad-setup}
\end{align}
A projected subgradient method requires $O(1/\varepsilon^2)$ subgradient evaluations to minimize a generic Lipschitz function. Applying Theorem~\ref{thm: main||-MRMtheorem}, we find a parallel multiradial subgradient method also only requires $O(1/\varepsilon^2)$ subgradient evaluations, up to a parallelizable factor of $N = O(\log(1/\varepsilon))$.
\begin{corollary}\label{cor: subgradRateWith||-MRM}
    Let Assumptions~\ref{assmpn: assumption1}~-~\ref{assmpn: assumption4} hold. If $\mathtt{fom}$ and $\varphi_\feasReg$ are set as in~\eqref{eq:subgrad-setup}, then, for all $\varepsilon > 0,$ Algorithm~\ref{alg: ||-MRM} with $N \geq \log_b(\frac{bp^*}{c_{\tau_0}\rho \varepsilon})$ finds a feasible $\varepsilon$-minimizer within $O(1/\varepsilon^2)$ iterations. 
\end{corollary}
Much like prior radial methods~\cite{Renegar2016,Grimmer2017}, no Lipschitz continuity assumptions are needed and projections are entirely avoided, instead just relying on linesearches and normal vectors. Unlike these prior radial methods, the usage of a common reference point $e\in\mathrm{int\ }\feasReg$ is avoided. Instead, separate centers are utilized, which also facilitates the potential for smaller Lipschitz constants for each gauge.

If, additionally, the objective $f$ is smooth and the sets $S_j$ are smooth and compact, accelerated methods can be applied. A set is $\beta$-smooth if its unit normal vectors are $\beta$-Lipschitz continuous on the set's boundary. Recently, \cite[Corollary 3.2]{liu2023gauges} showed every $\beta$-smooth compact set $S_j$ has $\frac{1}{2}\gamma_{S_j,e_j}^2$ as a $O(\beta)$-smooth function. Using this line of reasoning, one can also show that $f^{\Gamma, e_0}$ is $O(\beta)$-smooth if $f$ is $\beta$-smooth and $\dom f$ is compact. This motivates the usage of accelerated smoothing and generalized gradient methods applied with gauges squared occurring in the identifier. We consider the following two settings:
\begin{align}
    \mathtt{fom} = \mathtt{smooth} \qquad &\text{and} \qquad \varphi_{\feasReg}=\max\{\varphi_{S_1},\dots,\varphi_{S_m}\}\label{eq:smooth-setup}\\
    &\text{where}\quad \varphi_{S_j}(x) = \begin{cases} \gauge{S_j}{e_j}(x) & \text{if } \gauge{S}{e}(x) > 1 \\ \frac{1}{2}\gauge{S_j}{e_j}^2(x) + \frac{1}{2} & \text{otherwise} \end{cases} \ ,\nonumber  \\
    \mathtt{fom} = \mathtt{genGrad} \qquad &\text{and} \qquad \varphi_\feasReg = \max\set{\gauge{S_1}{e_1}^2, \dots, \gauge{S_m}{e_m}^2} \ .\label{eq:gengrad-setup} 
\end{align}
In both cases, up to a logarithmic term, these methods' $O(1/\varepsilon)$ and $O(1/\sqrt{\varepsilon})$ convergence rates are preserved, now providing new accelerated projection-free methods. For ease, our corollaries assume $f$ is twice continuously differentiable. However, this assumption is not needed because, as pointed out earlier, $f^{\Gamma, e_0}$ is smooth as long as $f$ is smooth and $\dom f$ is compact.
\begin{corollary}\label{cor: smoothRateWith||-MRM}
    Let Assumptions~\ref{assmpn: assumption1}~-~\ref{assmpn: assumption4} hold and $f$ be twice continuously differentiable. If $f$ is $\beta$-smooth and each $S_j$ is $\beta$-smooth and compact, and $\mathtt{fom}$ and $\varphi_\feasReg$ are set as in~\eqref{eq:smooth-setup}, then, for all $\varepsilon > 0,$ Algorithm~\ref{alg: ||-MRM} with $N \geq \log_b(\frac{bp^*}{c_{\tau_0}\rho \varepsilon})$ finds a feasible $\varepsilon$-minimizer within $O(1/\varepsilon)$ iterations.
\end{corollary}
\begin{corollary}\label{cor: genGradRateWith||-MRM}
    Let Assumptions~\ref{assmpn: assumption1}~-~\ref{assmpn: assumption4} hold and $f$ be twice continuously differentiable. If $f$ is $\beta$-smooth and each $S_j$ is $\beta$-smooth and compact, and $\mathtt{fom}$ and $\varphi_\feasReg$ are set as in~\eqref{eq:gengrad-setup}, then, for all $\varepsilon > 0,$ Algorithm~\ref{alg: ||-MRM} with $N \geq \log_b(\frac{bp^*}{c_{\tau_0}\rho \varepsilon})$ finds a feasible $\varepsilon$-minimizer within $O(1/\sqrt{\varepsilon})$ iterations.
    
\end{corollary}

\subsection{Practical Consideration} \label{subsec:practical}
To apply $||\mbox{-}\mathtt{MRM}$ with $\varphi_\feasReg$ constructed from gauges (squared) requires three main ingredients, computing the reference points $e_j$ each interior to the related constraint $S_j$, computing a feasible initialization $x_0$ and $\tau_0 = 1/f(x_0)>0$, and computing function values and subgradients of $f^{\Gamma,e_0}$ and $\gamma_{S_j,e_j}$ for the underlying first-order method. Below, we address these three computations and provide an extension to allow affine constraints (which have no interior and so are beyond the scope of Assumption~\ref{assmpn: assumption1}).

\paragraph{Computing Selection of $e_j$} Our multiradial duality theory avoids a reliance on knowing a good reference point interior to $\cap S_j$ (with the quality measured by $R_e(\cap S_j)$). Instead, points $e_j$ with reasonably positive $R_{e_j}(S_j)$ are needed. One natural choice of $e_j$ is the Chebyshev center, defined as maximizing $e \mapsto R_{e}(S_j)$. For generic convex $S_j$, computing this is a convex optimization problem. For polyhedrons, this corresponds to an LP. For  norm-type constraints $\{x \mid \|A_jx - b_j\|\leq 1 \}$, its center is given by any solution to $A_jx=b_j$.

For our numerics, we consider QCQPs where the center is also given by a linear system solve. Our results in Section~\ref{sec:numerics} show that in this setting, the choice of centers has no observable effect on convergence. Therefore, exact solutions to the systems $A_jx = b_j$ are not needed. One could, for instance, compute the centers $e_j$ using only a few conjugate gradient steps.
Note for a given $e_j$, computing or estimating $R_{e_j}(S_j)$ is nontrivial. For convex QCQPs~\eqref{eq:qcqp}, this amounts to a nonconvex QCQP.

\paragraph{Computing an initialization and rescaling $\tau_0$}
Given a selection $e_j$, $||\mbox{-}\mathtt{MRM}$ still requires knowledge of a sufficiently large $\tau_0$ such that $p(\tau_0) \geq 1$. This can be done directly by finding any $x_0\in \cap_{j=0}^m S_j$ and setting $\tau_0 = 1/f(x_0)$. Such a point can be found by minimizing the maximum of the gauges of each $S_j$ with respect to $e_j$ until a value less than one is reached (which the Slater point ensures is possible). Noting that $\lim_{\tau\rightarrow \infty} (\tau f)^{\Gamma,e_0}(y) = \gamma_{S_0,e_0}(y)$, computing an initial feasible point in the domain of $f$ can be viewed as approximately minimizing $\Phi_{\infty}(y)\defeq \lim_{\tau\rightarrow\infty} \Phi_\tau(y)$.
Hence the cost of adding such a first phase to bootstrap $||\mbox{-}\mathtt{MRM}$ is comparable to the cost of approximately minimizing one subproblem $\Phi_\tau$.

\paragraph{Computing $f^{\Gamma,e_0}$ and $\gamma_{S_j,e_j}$ (and their subgradients)} Often $f^\Gamma$ and $\gamma_S$ have closed forms (see~\cite[Tables 1 and 2]{radial2}) and their (sub)gradients can be directly computed from (sup)gradients and normal vectors of $f$ and $S$ (see~\cite[Proposition 19 and 21]{radial1}). For example, generic polyhedral constraints $\{x \mid Ax\leq b\}$ or ellipsoidal constraints $\{x\mid \|Ax-b\|_2\leq 1\}$ have closed forms for their gauge, computable by a single matrix-vector multiplication, see~\eqref{eq:dual-qcqp}.

If a closed form is not available, evaluating the radial transformation of a function or the gauge of a set amounts to a one-dimensional linesearch. Given a function value oracle for $f$ or membership oracle for $S_j$, this can be computed by any root-finding methods (e.g., bisection). Algorithms based on such inexact evaluations were developed by the works~\cite{Zakaria2022,Lu2023}.

\paragraph{Reformulations with Affine Constraints} As stated, our multiradial duality theory does not directly apply to problems with affine constraints $A x = b$ among the set constraints $x\in\cap_{j=1}^m S_j$ since the affine constraints have no interior (and hence cannot satisfy Assumption~\ref{assmpn: assumption1}). Such constraints can be addressed separately from $S_j$ by additionally requiring that each $e_j$ satisfies $Ae_j = b$, then consider the affine constrained primal and multiradial dual functions
$$ \Psi(x) = \begin{cases}
    f(x) & \text{if } Ax=b, \  x\in\feasReg\\
    0 & \text{otherwise}
\end{cases}
\quad
\Phi(y) = \begin{cases}
    \max\{f^{\Gamma,e_0}(x), \gamma_{S_j,e_j}(y)\} & \text{if } Ay=b\\
    \infty & \text{otherwise .}
\end{cases}
$$
Our Theorems~\ref{thm: primialDualEqualityAtOne}, \ref{thm: dualUpperBoundIfPrimalAboveOne}, and~\ref{thm: primalLowerBoundIfDualBelowOne} directly generalize to this setting which restricts to the affine subspace where $Ax=b$, relating maximizers of $\Psi$ and minimizers of $\Phi$. Consequently, given a first-order method capable of minimizing a finite maximum over affine constraints, $||\mbox{-}\mathtt{MRM}$ could be applied to solve an affine-constrained primal. For example, by precomputing the projection operator onto the affine space, a projected subgradient method could be applied, while remaining projection-free with respect to the more sophisticated $S_j$ constraints.

    \section{Numerical Validation}\label{sec:numerics}
In this final section, we apply our theory to synthetically generated QCQP problems. Our primary goal is to validate our theoretical guarantees for $\pMRM$ working in parameter-free fashion ``out-of-the-box'' and highlight a surprising disconnect where performance outscales our theory's predictions. Our implementation is not state-of-the-art, and so we restrict our attention to understanding $\pMRM$ rather than comparisons with other methods. We consider QCQPs of the form
\begin{equation}\label{eq: qcqp_recalled}
    p^* = \begin{cases}
        \max & f_0(x) := r_0 - q_0^Tx - \frac{1}{2} x^TP_0x \\
        \mathrm{s.t.} & f_j(x) := r_j - q_j^Tx - \frac{1}{2} x^TP_jx \geq 0 \quad \forall j=1, \dots, m \ .
    \end{cases}
\end{equation}
where the matrices $P_j$ are symmetric and positive definite.

All our synthetic problems are constructed as follows. The matrices $P_j$ take the form $P_j = G_j^TG_j + \lambda I$ for all $j = 0, 1, \dots, m$, where $\lambda = 0.01,$ $I \in \R^{n \times n}$ is the identity matrix, and each entry of $G_j \in \R^{n \times n}$ is sampled independently from the standard normal distribution. Each $q_j$ is drawn independently from the normal distribution with mean $0$ and covariance $\sigma_j I.$ To avoid the trivial case where the solution is interior to the constraints, we take $\sigma_0 = 10$ and $\sigma_j = 1$ for $j \geq 1.$ Finally, to guarantee a Slater point exists, we ensure $f_j(0) > 0$ by selecting $r_j$ independently and uniformly from $[0.1, 1.1].$ In all cases, $\varspace = \mathbb{R}^{200}$ with the standard Euclidean norm. Code implementing these experiments can be found at \url{https://github.com/samaktbo/Parallel-MultiRadial-Method}.

\subsection{Performance of MRM with Varied Subproblem Solvers}
First, we investigate how the $\pMRM$ method performs under different first-order solvers. Specifically, for each $\mathtt{fom} \in \set{\mathtt{subgrad}, \mathtt{smooth}, \mathtt{genGrad}},$ and $m\in\set{10,100,1000}$, Figure~\ref{fig:convergence_vs_time} shows the relative optimality gap $\frac{p^* - f_0(y^{best}_i)}{p^* - f_0(x_0)}$ varies as a function of real-time and the number of iterations. We initialize each method with $x_0 = 0$ and Algorithm~\ref{alg: ||-MRM} with $b = 4.0$ and $N = 16$ parallel instances. We select the centers in an ideal fashion, i.e., we set $e_j = - P_j^{-1}q_j,$ for each $j = 0, 1, \dots, m.$ We see that for relatively small number of constraints ($m \leq 100$), the generalized gradient method far outperforms the theoretical per-iteration convergence rate of $O(1/\sqrt{\varepsilon}).$  However, this method scales poorly since each iteration requires a QP solve from \texttt{Mosek}, completing about 30 iterations in 3000 seconds for $m = 1000.$ On the other hand, the smoothing method and the subgradient method scale reasonably with $m,$ even though their rate of convergence is slower. These method's convergence matches their theoretically predicted rates of $ O(1/\varepsilon)$ and $O(1/\varepsilon^2),$ respectively, after slower convergence in the first hundred or so iterations, potentially corresponding to the $K^{(l_0)}_{\mathtt{fom}}$ term in Theorem~\ref{thm: main||-MRMtheorem}.

\begin{figure}
    \centering
    \includegraphics[width=0.3\textwidth]{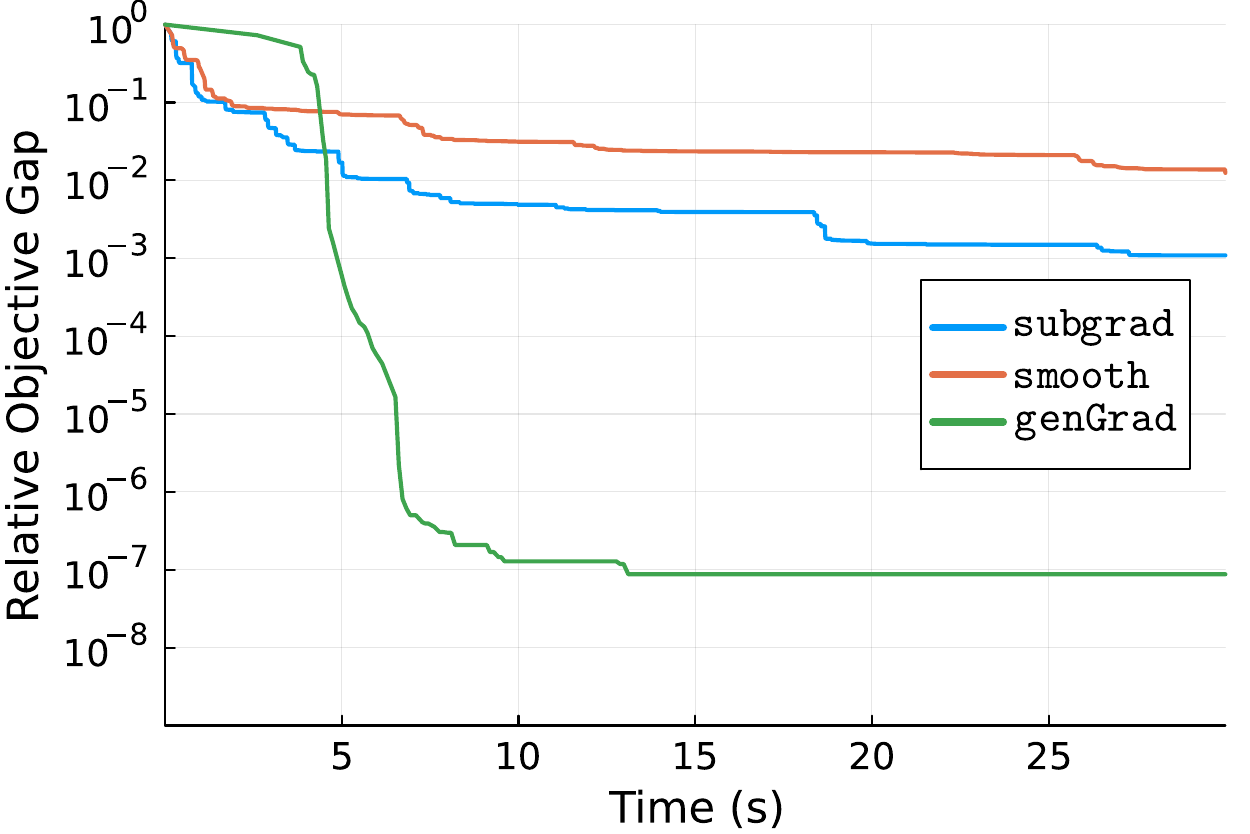}
    \hfill
    \includegraphics[width=0.3\textwidth]{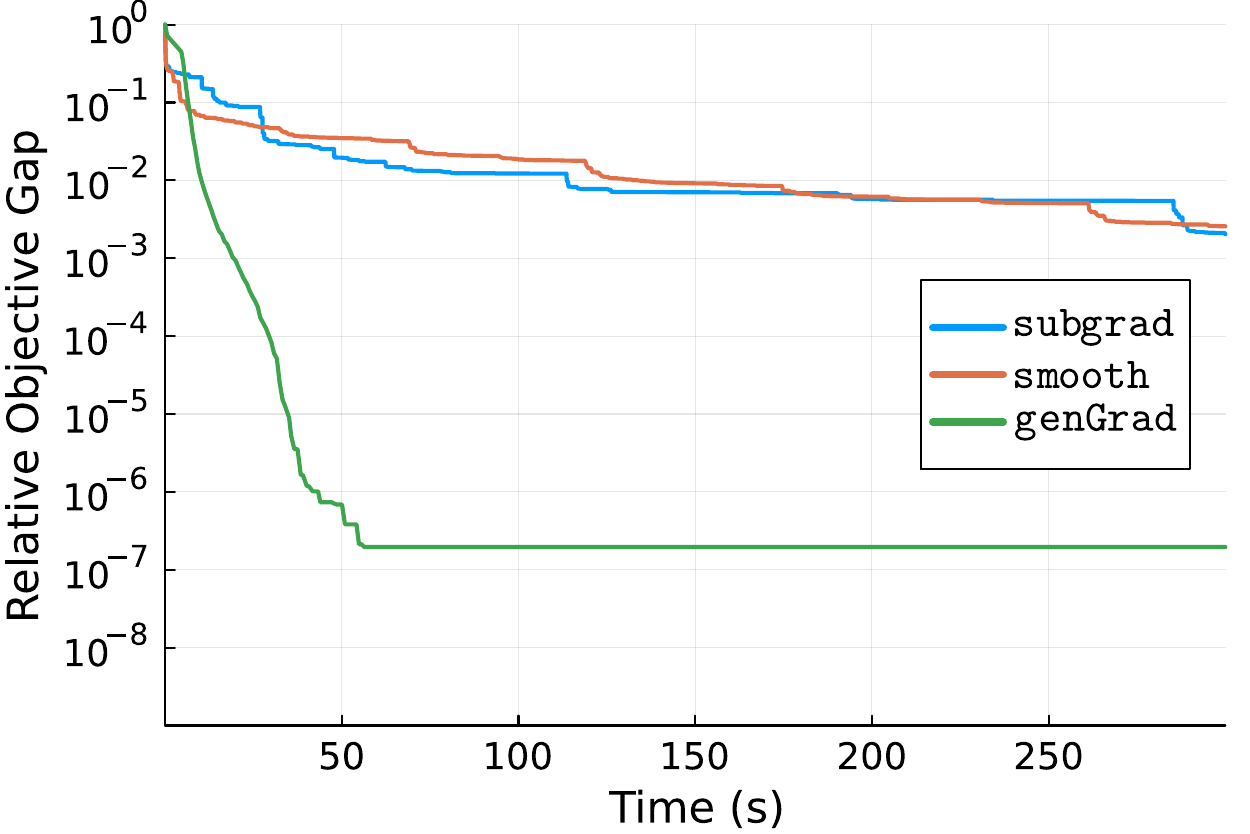}
    \hfill
    \includegraphics[width=0.3\textwidth]{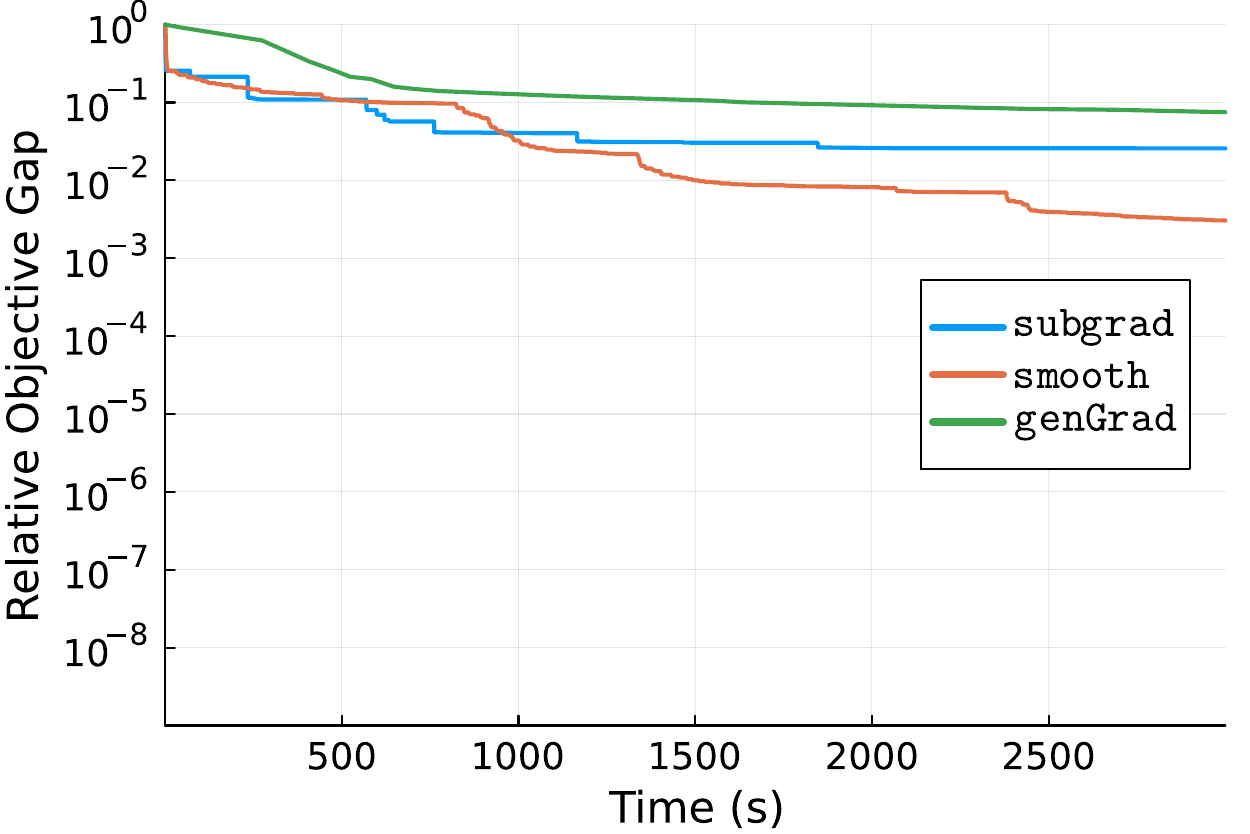}
    \includegraphics[width=0.3\textwidth]{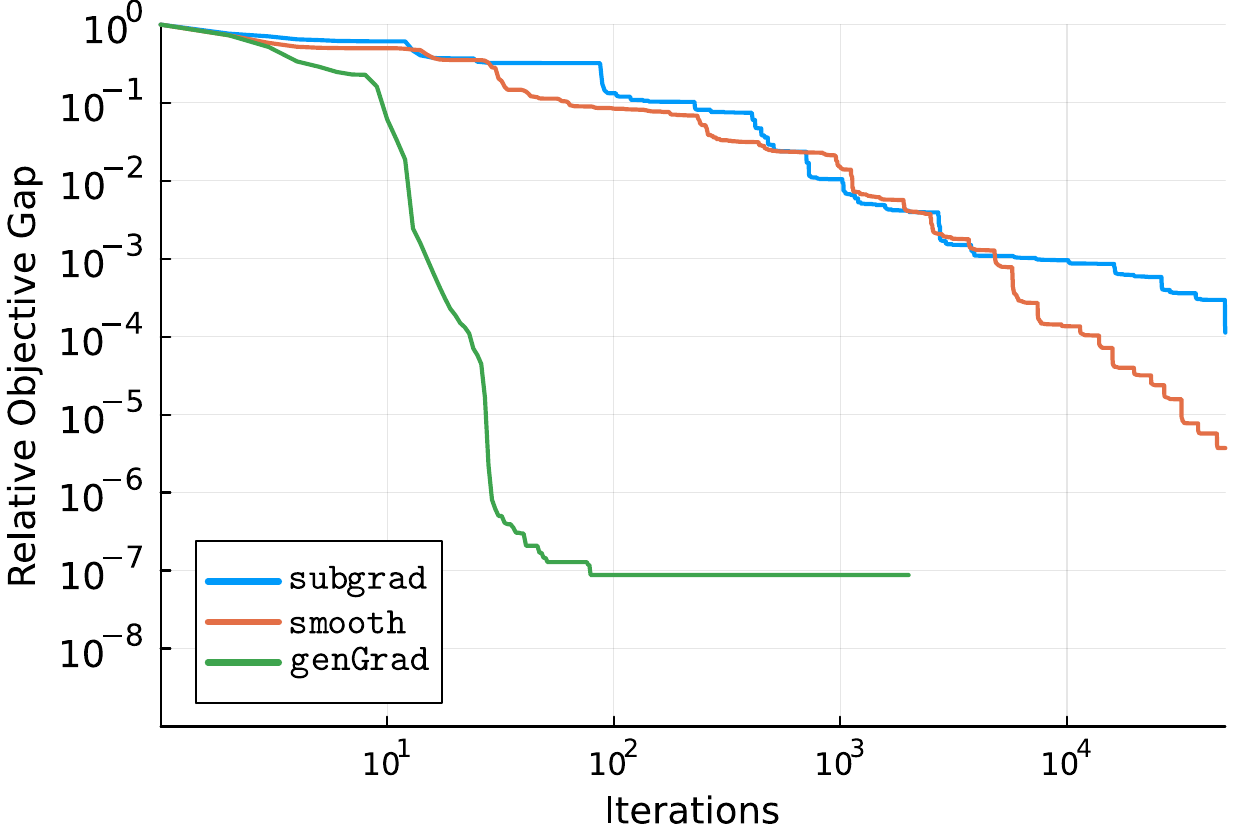}
    \hfill
    \includegraphics[width=0.3\textwidth]{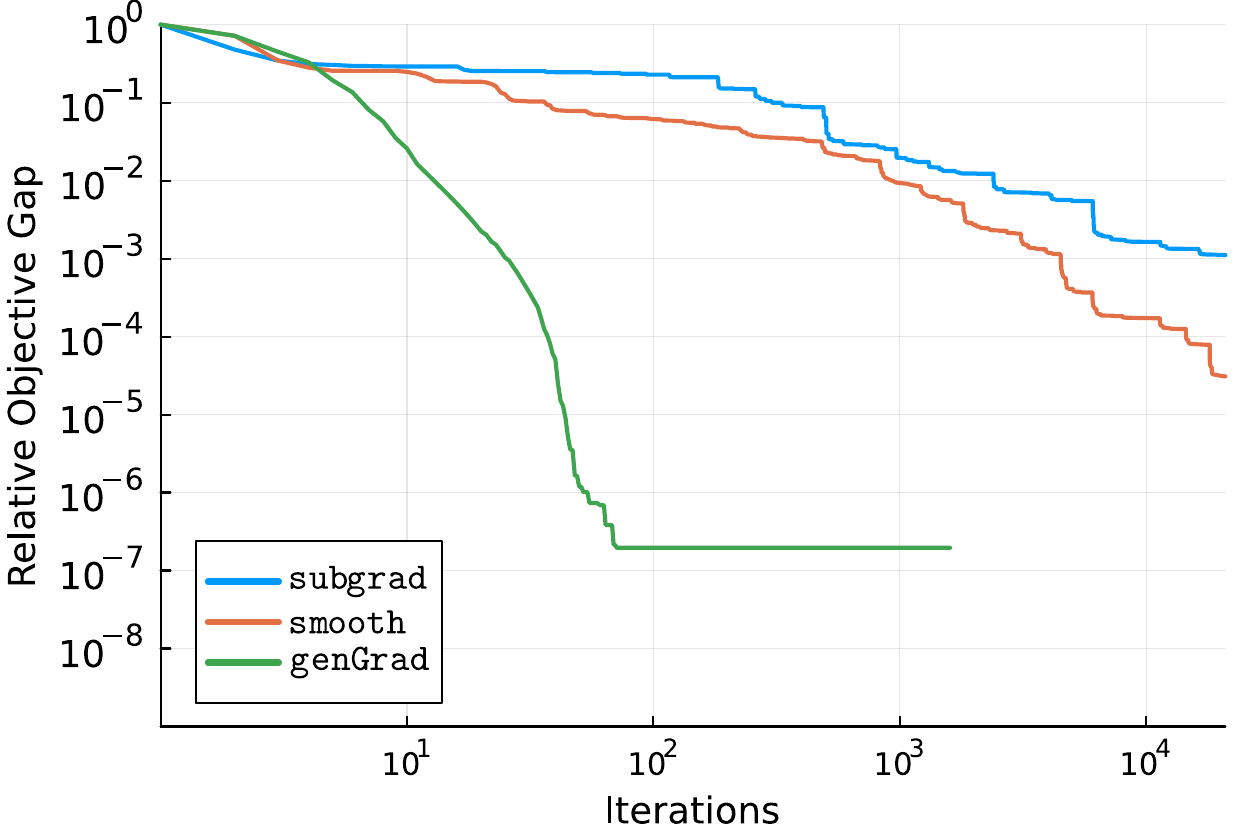}
    \hfill
    \includegraphics[width=0.3\textwidth]{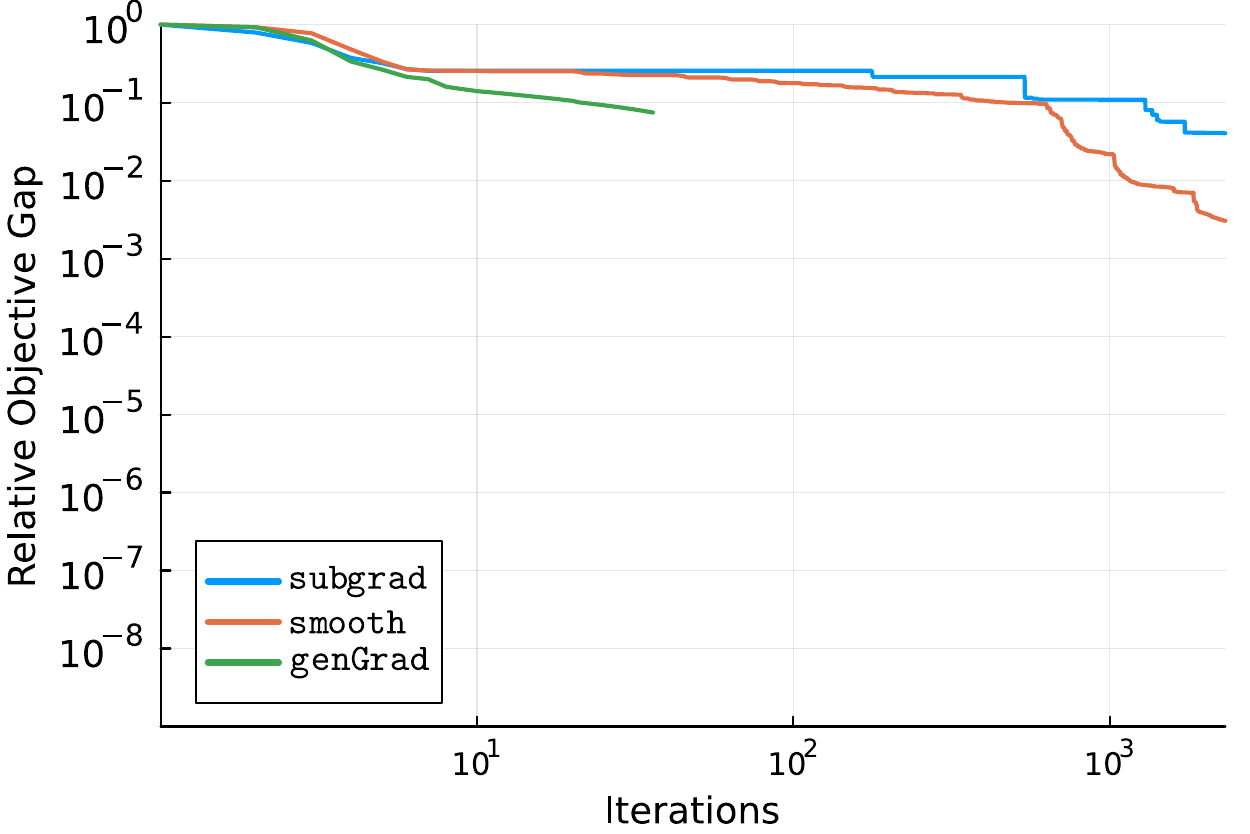}
     
    \caption{Performance of $\pMRM$ utilizing each of $\set{\mathtt{subgrad}, \mathtt{smooth}, \mathtt{genGrad}}$ in relative optimality gap $\frac{p^* - f_0(y^{best}_T)}{p^* - f_0(x_0)}$, plotted against real-time and number of iterations. The number of constraints is $m = 10,$ $m = 100,$ and $m = 1000,$ from left to right, respectively.}
    \label{fig:convergence_vs_time}
\end{figure}

\subsection{Effects of Multiradial Centers On Convergence}
Next, we examine how the performance of Algorithm~\ref{alg: ||-MRM} is affected by the choice of centers $e_0, e_1, \dots, e_m$ for problems of the form \eqref{eq: qcqp_recalled}. We utilize the same set of underlying first-order methods and sample three QCQPs for the same selections of $m$ as before. Then, for $K=300$ target magnitudes of $R$ ranging from $10^{-6}$ to $10^{-1}$, we randomly sample centers $e_j$ with controlled $R_{e_j}(S_j)$ (see full construction below). Surprisingly, Figure~\ref{fig:generic_R_performance} shows the (relative) optimality gap of the iterates of Algorithm~\ref{alg: ||-MRM} reached is essentially independent of the choice of centers $e_j$ and the related constant $R$. Practically, this indicates one need not spend much computational effort to find ``good'' centers to use. Conceptually, this indicates a gap between our theoretical bounds and actual performance. This is true for all three solvers and across problem sizes.

Note for a given $e_j$, computing $R_{e_j}(S_j)$ is a nontrivial nonconvex optimization problem. To avoid this difficulty, rather than randomly sampling $e_j$ and computing the resulting $R$, we generate the $e_j$ in such a way that $R_{e_j}(S_j)$ has a closed form. Namely for any $x_j$ with $f_j(x_j) = 0$ and $0< \alpha \leq 1,$ $e_j = x_j + \frac{\alpha}{\|P_j\|}\nabla f_j(x_j)$ has $f_j(e_j) > 0$ and $R_{e_j}(S_j) = \alpha \frac{\|\nabla f_j(x_j)\|}{\|P_j\|}$ since $f_j$ is $\|P_j\|$-smooth. 
For each of our $K=300$ trials, we use this construction for $e_j,$ setting $\alpha$ uniformly between $0.01$ and $1$ (in log-scale) and $x_j = \Bar{e}_j + \sqrt{2f(\Bar{e}_j)}P_j^{-\frac{1}{2}}u_j$ where $\Bar{e}_j = -P_j^{-1}q_j$ and $u_j$ is sampled uniformly from the unit sphere. The scaling $\sqrt{2f(\Bar{e}_j)}$ ensures that $f(x_j) = 0.$

\begin{figure}
    \centering
    \includegraphics[width=0.3\textwidth]{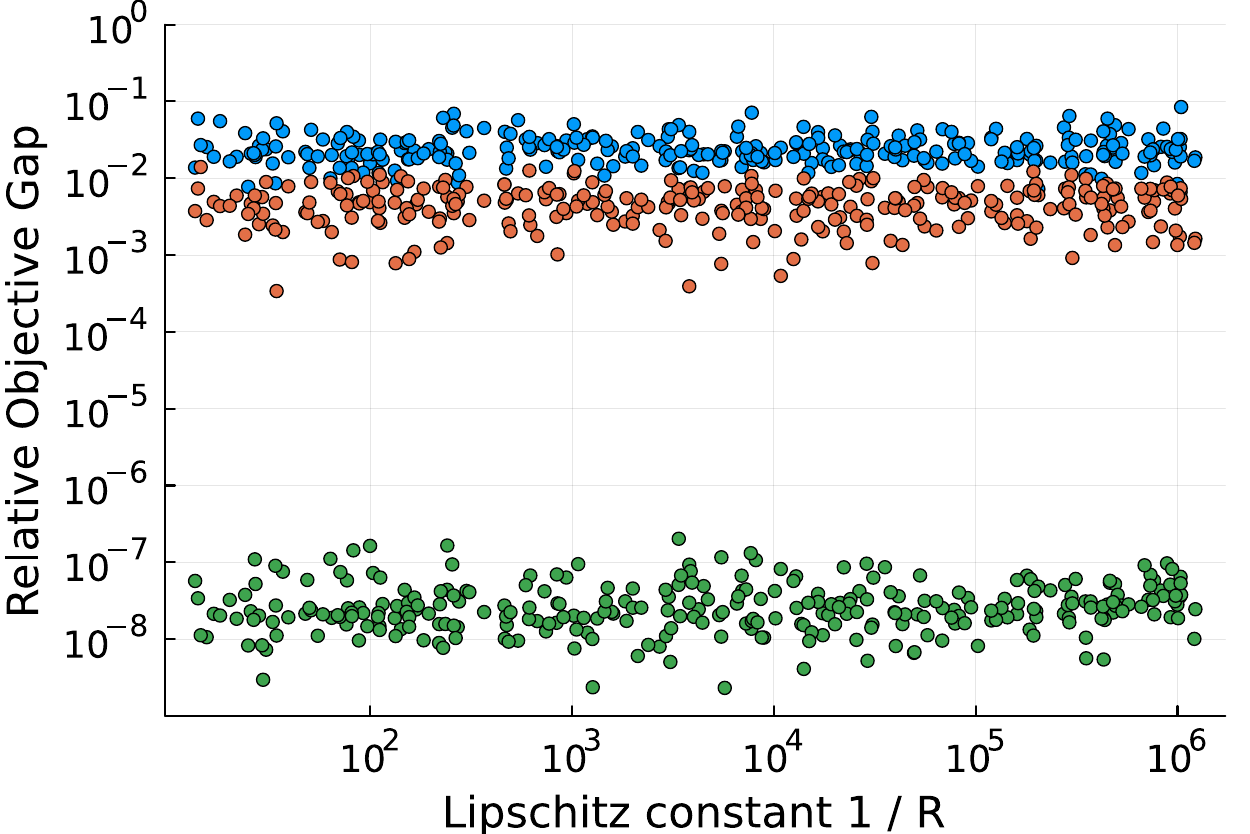}
    \hfill
    \includegraphics[width=0.3\textwidth]{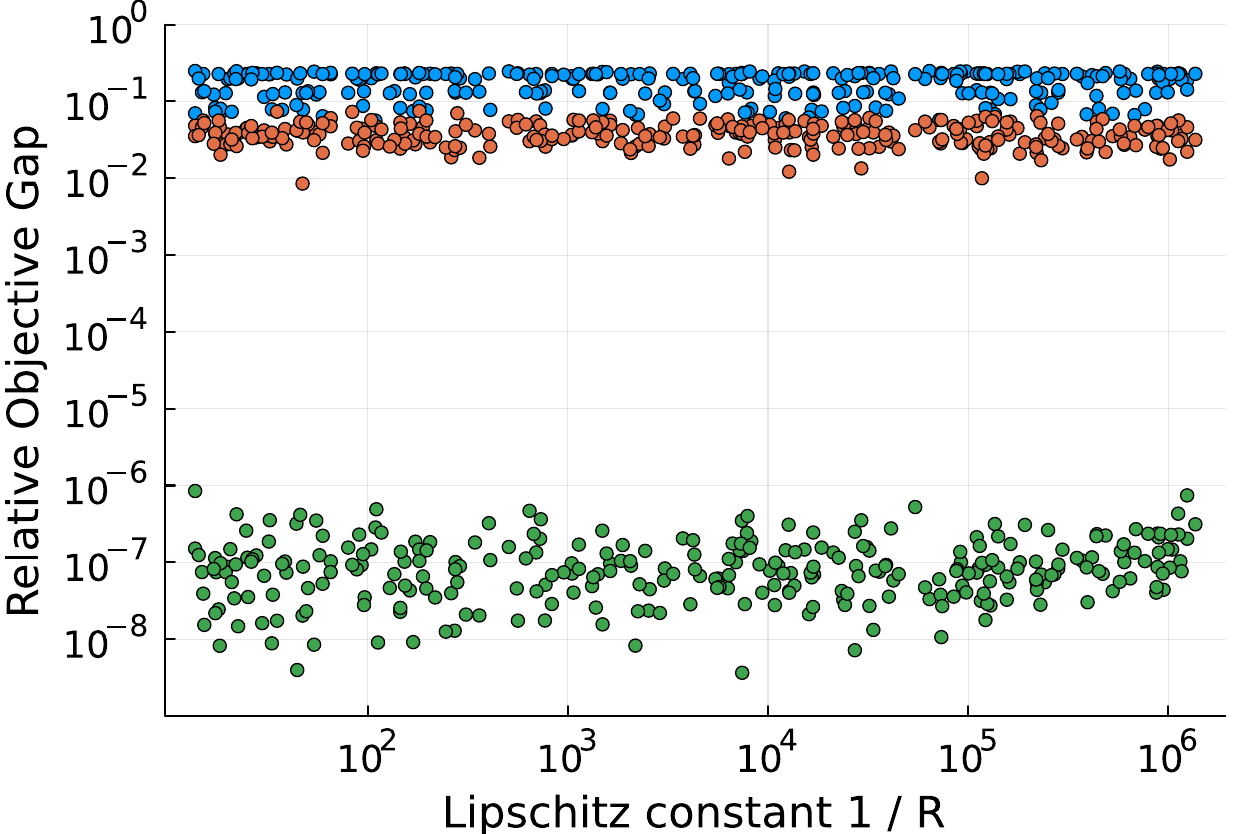}
    \hfill
    \includegraphics[width=0.3\textwidth]{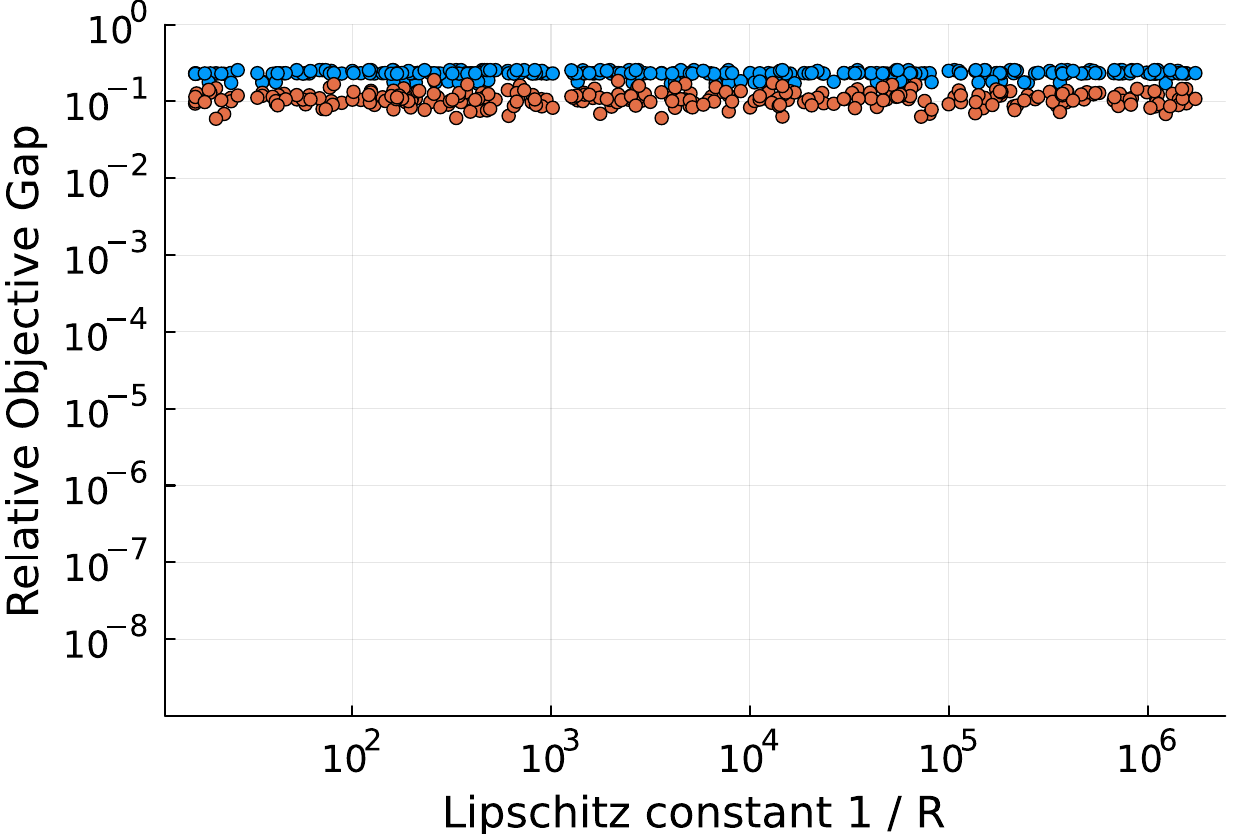}
    \caption{Final relative optimality gap $\frac{p^* - f_0(y^{best}_T)}{p^* - f_0(x_0)}$ vs Lipschitz constant for three solvers: $\mathtt{subgrad}$~(blue), $\mathtt{smooth}$~(orange), and $\mathtt{genGrad}$~(green). The number of constraints is $m = 10,$ $m = 100,$ and $m = 1000,$ from left to right. For $m = 1000,$ $\mathtt{genGrad}$ was prohibitively costly to run.}
    \label{fig:generic_R_performance}
\end{figure}


    \section{Deferred Analysis of Parallel MultiRadial Method} \label{sec:proofs}

We begin by bounding the rate at which the dual gaps $1 - d(\tau^{best}_{i})$ decrease. Recall that $\mathtt{fom}^{(l)}$ of Algorithm~\ref{alg: ||-MRM} restarts at $i\geq 0$ if $\tau^{best}_{i + 1} \leq \frac{1}{1 + \delta^{(l)}}\tau^{(l)}_i$ and $i^{(l)}_0<i^{(l)}_1<i^{(l)}_2\dots $ denotes the sequence of iterations where the $l$th first-order method instance restarted.

\begin{lemma}\label{lem: dualGapDecayForThe||-MRM}
    Under Assumptions~\ref{assmpn: assumption1}~-~\ref{assmpn: assumption4}, if $\frac{b \delta^{(l)}}{\rho} \leq 1 - d(\tau^{best}_i)$ at iteration  $i \geq 0$ of Algorithm~\ref{alg: ||-MRM} with $b \geq 2$, then 
    $$ 
    1 - d(\tau_{i'}^{best}) < \frac{b \delta^{(l)}}{\rho} \quad \text{for all} \quad i' > i + \frac{1}{\log(1 + \delta^{(l)})}\log\left(\frac{p(\tau^{(l)}_i)}{1 + b\delta^{(l)}}\right)K_{\mathtt{fom}}^{(l)}.
    $$ 
\end{lemma}
\begin{proof}
    It suffices to bound the number of iterations until $p(\tau^{best}_{i'}) < 1 + b\delta^{(l)}$ holds since by Theorem~\ref{thm: primalLowerBoundIfDualBelowOne}, this implies $1 - d(\tau^{best}_{i'}) < \frac{b\delta^{(l)}}{\rho}$. Let $i^{(l)}_{k}$ denote the first iteration after $i$ where $\mathtt{fom}^{(l)}$ restarts and $i^{(l)}_{\hat k}$ denote the first iteration with $p(\tau^{(l)}_{i^{(l)}_{\hat k}}) < 1 + b\delta^{(l)}$. The restarting condition of $\mathtt{fom}^{(l)}$ ensures $p(\tau^{(l)}_{i^{(l)}_{j+1}}) \leq (1+\delta^{(l)})^{-1}p(\tau^{(l)}_{i^{(l)}_{j}})$. Therefore
    $$ p(\tau^{best}_{i^{(l)}_{\hat k}}) \leq p(\tau^{(l)}_{i^{(l)}_{\hat k}}) \leq (1+\delta^{(l)})^{-(\hat k - k)}p(\tau^{(l)}_{i}) \ . $$
    Hence after $\frac{1}{\log(1 + \delta^{(l)})}\log\left(\frac{p(\tau^{(l)}_i)}{1 + b\delta^{(l)}}\right)$ restarts of $\mathtt{fom}^{(l)}$, every iteration $i'$ must have $p(\tau^{best}_{i'}) < 1 + b\delta^{(l)}$ and hence $1 - d(\tau^{best}_{i'}) < \frac{b\delta^{(l)}}{\rho}$.
    
    All that remains to bound the number of iterations between consecutive restarts of $\mathtt{fom}^{(l)}$ by $K_{\mathtt{fom}}^{(l)}$.
    Consider some pair of restart times $i^{(l)}_k < i^{(l)}_{k+1}$ with $i^{(l)}_k \geq i.$ If some first-order method instance $l'\neq l$ at an iteration $i' \leq i^{(l)}_k + K_{\mathtt{fom}}^{(l)}$, finds an iterate improving $\tau^{best}$ to be less than $\tau^{(l)}_{i^{(l)}_k}/(1+\delta^{(l)})$, then $\mathtt{fom}^{(l)}$ will restart with $i^{(l)}_{k+1} \leq i^{(l)}_k + K_{\mathtt{fom}}^{(l)}$. Otherwise, $\mathtt{fom}^{(l)}$ proceeds without interruption from other processes for at least $K_{\mathtt{fom}}^{(l)}$ iterations. Then, by definition, some $i' \leq i^{(l)}_k + K_{\mathtt{fom}}^{(l)}$ has $y^{(l)}_{i'}$ be a $\delta^{(l)}/\rho$-minimizer of $\Phi_{\tau^{(l)}_{i}}$. Since $\frac{\delta^{(l)}}{\rho} \leq (1 - d(\tau^{best}_{i}))/b \leq  (1 - d(\tau^{(l)}_{i^{(l)}_k}))/b$, Corollary~\ref{cor: fomWorksIfDeltaSmall} implies $y^{(l)}_{i'}$ is feasible and $\tau^{(l)}_{i^{(l)}_k} f(y^{(l)}_{i'}) - 1 \geq (1 - 1/b)\rho (1- d(\tau^{(l)}_{i^{(l)}_k})) \geq \delta^{(l)}$. Hence $1/f(y^{(l)}_{i'}) \leq \frac{1}{1+\delta^{(l)}}\tau^{(l)}_{i^{(l)}_k}$ and so $i^{(l)}_{k+1} \leq i^{(l)}_k + K_{\mathtt{fom}}^{(l)}$.

\end{proof}

\subsection{Proof of Theorem~\ref{thm: main||-MRMtheorem}}
    From Lemma~\ref{lem: dualGapDecayForThe||-MRM}, we arrive at the following.
    \begin{theorem}\label{thm: fundamentalThmOf||-MRM}
    Suppose Assumptions~\ref{assmpn: assumption1}~-~\ref{assmpn: assumption4} hold and let $\mathtt{fom}$ be given. Then, for all $\Bar{\varepsilon} > 0$, Algorithm~\ref{alg: ||-MRM} with $\mathtt{fom},$ $b \geq 2,$ and $N = \lceil\log_b(1/\Bar{\varepsilon})\rceil$ has 
    $$
    1 - d(\tau^{best}_i) \leq \frac{b\Bar{\varepsilon}}{\rho} \quad \text{for all} \quad i \geq \frac{\log(\tau_0 p^*)}{\log(1 + \frac{\rho}{b^2})} K^{(l_0)}_{\mathtt{fom}} + \frac{5b^2(1 + \frac{\rho}{b^2})}{4\rho c_{\tau_0}}\sum_{l = l_0 + 1}^{N} K^{(l)}_{\mathtt{fom}}
    $$
    In fact, any such $i$ has $1 - d(\tau^{best}_i)\leq \frac{b\delta^{(N)}}{\rho}.$
    \end{theorem}
    \begin{proof}
    Note that $\frac{b\delta^{(N)}}{\rho} = \frac{b}{b^{N}\rho} \leq \frac{b\Bar{\varepsilon}}{\rho}$ so that the second statement of the theorem implies the first. Now, if $1 - d(\tau_0) \leq \frac{b\delta^{(N)}}{\rho}$ then there is nothing to prove. We therefore assume for the rest of the proof that $\frac{b}{b^{N}\rho} < 1 - d(\tau_0) \leq 1.$ 

    Considering the partition of $[0,1]$ given by $\set{0, \frac{b}{b^{N}\rho}, \frac{b}{b^{N-1}\rho}, \dots, \frac{b}{b^{l_0}\rho}, 1}$, Lemma~\ref{lem: dualGapDecayForThe||-MRM} gives a bound on how long $1 - d(\tau^{best}_i)$ can remain in each sub-interval of $(\frac{b}{b^{N}\rho}, 1]$. We get the number of iterations needed to have $1 - d(\tau^{best}_i) \leq \frac{b}{b^{N}\rho}$ by summing the total number of iterations needed to move $1 - d(\tau^{best}_i)$ out of each sub-interval of $(\frac{b}{b^{N}\rho}, 1]$.
    
    Note that $\frac{\log(p(\tau_0))}{\log(1 + \frac{\rho}{b^2})}K_{\mathtt{fom}}^{(l_0)} \geq \frac{1}{\log(1 + \delta^{(l_0)})}\log\left(\frac{p(\tau_0)}{1 + b\delta^{(l_0)}}\right)K_{\mathtt{fom}}^{(l_0)}$ since $\rho \leq b/b^{l_0 - 1}$ by the definition of $l_0$. Therefore, if $\frac{b\delta^{(l_0)}}{\rho} < 1 - d(\tau_0),$ then $1 - d(\tau^{best}_i) < \frac{b\delta^{(l_0)}}{\rho}$ for all $i > \frac{\log(p(\tau_0))}{\log(1 + \frac{\rho}{b^2})}K_{\mathtt{fom}}^{(l_0)},$ by Lemma~\ref{lem: dualGapDecayForThe||-MRM}. 
    
    Note that the restarting condition implies $\tau^{(l)}_i \leq (1 + \delta^{(l)})\tau^{best}_i.$ Thus, any $i$ with $1 - d(\tau^{best}_i) \leq \frac{b^2\delta^{(l)}}{\rho}$ has
    \begin{align*}
        p(\tau^{(l)}_i) \leq (1 + \delta^{(l)}_i)p(\tau^{best}_i) \leq 1 + \frac{1 - d(\tau^{best}_i)}{c_{\tau^{best}_i}} + p(\tau^{best}_i)\delta^{(l)} 
        &\leq 1 + \left[\frac{b^2}{\rho c_{\tau^{best}_i}} + p(\tau^{best}_i)\right]\delta^{(l)} \\
        &\leq 1 + \left[\frac{b^2}{\rho c_{\tau_0}} + p(\tau_0)\right]\delta^{(l)}, 
    \end{align*}    
    where the second inequality is by Theorem~\ref{thm: dualUpperBoundIfPrimalAboveOne}, the third is by assumption, and the fourth holds because $\tau^{best}_i \leq \tau_0.$ Therefore,
    \begin{align*}
        \frac{1}{\log(1 + \delta^{(l)})}\log\left(\frac{p(\tau^{(l)}_i)}{1 + b\delta^{(l)}}\right) 
        &= \frac{1}{\ln(1 + \delta^{(l)})}\ln\left(1 + \frac{p(\tau^{(l)}_i) - 1 - b\delta^{(l)}}{1 + b\delta^{(l)}}\right) \\
        &\leq \frac{1}{\ln(1 + \delta^{(l)})}\left(\frac{p(\tau^{(l)}_i) - 1 - b\delta^{(l)}}{1 + b\delta^{(l)}}\right) \\
        &\leq \frac{1}{\ln(1 + \delta^{(l)})}\left[\frac{b^2}{\rho c_{\tau_0}} + p(\tau_0) - b\right]\frac{\delta^{(l)}}{1 + b\delta^{(l)}} \\
        &\leq \frac{\delta^{(l)}}{\ln(1 + \delta^{(l)})}\frac{b^2}{\rho c_{\tau_0}}\left[1 + \rho\frac{c_{\tau_0}[p(\tau_0) - 1]}{b^2}\right] \\
        &\leq \frac{5b^2}{4\rho c_{\tau_0}}\left(1 + \frac{\rho}{b^2}\right)
    \end{align*}
    where the first inequality follows from $\ln(1 + x) \leq x$ and the last follows because $\frac{\delta}{\ln(1 + \delta)} \leq \frac{5}{4}$ for any $\delta \in (0, 1/2]$ and $c_{\tau_0}[p(\tau_0) - 1] \leq 1$ by Theorem~\ref{thm: dualUpperBoundIfPrimalAboveOne}. As such, if $\frac{b\delta^{(l)}}{\rho} < 1 - d(\tau^{best}_i) \leq \frac{b^2\delta^{(l)}}{\rho}$ for $l > l_0,$ then $1 - d(\tau^{best}_{i'}) \leq \frac{b\delta^{(l)}}{\rho}$ for all $i' > i + \frac{5b^2(1 + \frac{\rho}{b^2})}{4\rho c_{\tau_0}}K_{\mathtt{fom}}^{(l)}$ by Lemma~\ref{lem: dualGapDecayForThe||-MRM}. Summing everything completes the proof.
    \end{proof}

    From this theorem, our originally claimed Theorem~\ref{thm: main||-MRMtheorem} follows directly: Given $\varepsilon$ and $i$ as in Theorem~\ref{thm: main||-MRMtheorem}, consider $\Bar{\varepsilon} = \frac{c_{\tau_0}\rho \varepsilon}{bp^*}$. Then the above result ensures
    \begin{align*}
        \frac{b\Bar{\varepsilon}}{\rho} &\geq 1 - d(\tau^{best}_i) & [\text{Theorem~\ref{thm: fundamentalThmOf||-MRM}}] \\
        &\geq c_{\tau^{best}_i}(\tau^{best}_i p^* - 1) & [\text{Theorem~\ref{thm: dualUpperBoundIfPrimalAboveOne}}]\\
        &\geq c_{\tau_0}\frac{p^* - f(y^{best}_i)}{f(y^{best}_i)} & [c_{\tau_0} \leq c_{\tau^{best}_i}]\\
        &\geq c_{\tau_0}\frac{p^* - f(y^{best}_i)}{p^*}.
    \end{align*}
    Since $y^{best}_i$ is always feasible, the proof is complete.

    \paragraph{Acknowledgements.} This work was supported in part by the Air Force Office of Scientific Research under award number FA9550-23-1-0531.
    
    {\small
    \bibliographystyle{unsrt}
    \bibliography{bibliography}
    }
    \appendix
    \section{Proofs of Corollaries}\label{app:proofs}
Throughout this appendix, we let
\begin{align*}
    D \defeq \max\{D(S_j) \mid j = 0, 1, \dots, m\} \quad \text{and} \quad R \defeq \min\{R_{e_j}(S_j) \mid j = 0, 1, \dots, m\},
\end{align*}
where $e_1, \dots, e_m$ are the reference points defining $\gauge{S_1}{e_1}, \dots, \gauge{S_m}{e_m}.$

\subsection{Proof of Corollary~\ref{cor: subgradRateWith||-MRM}}
    Let $\Bar{\varepsilon} = \frac{c_{\tau_0}\rho \varepsilon}{bp^*}$ and $\tilde{N} = \lceil\log_b(\frac{bp^*}{c_{\tau_0}\rho \varepsilon})\rceil.$ By Theorem~\ref{thm: main||-MRMtheorem}, it suffices to show that $K^{(l_0)}_{\mathtt{subgrad}}$ and $\sum_{l = l_0 + 1}^{\Tilde{N}} K^{(l)}_{\mathtt{subgrad}}$ are bounded by $O(1 / \varepsilon^2).$ Since $\Phi_{\tau}$ is $1/ R$-Lipschitz for all $\tau > 0,$ it follows that $K^{(l)}_{\mathtt{subgrad}} \leq \frac{(D_0/R)^2}{(\delta^{(l)}/\rho)^2} = (\rho b^l D_0/R)^2.$ Now, we have $b^{l_0} \leq b^2/\rho$ by definition, hence $K^{(l_0)}_{\mathtt{subgrad}} \leq b^4\frac{D_0^2}{R^2}$ is constant with respect to $\varepsilon$. For $l > l_0,$ we have $b^l = \frac{b\cdot b^{\Tilde{N} - 1}}{b^{\Tilde{N} - l}} \leq \frac{b}{\Bar{\varepsilon}b^{\Tilde{N} - l}} = \frac{b}{\rho}\left(\frac{bp^*}{c_{\tau_0}}\right)\frac{1}{b^{\Tilde{N} - l}}\frac{1}{\varepsilon}.$ Therefore
    \begin{align*}
        \sum_{l = l_0 + 1}^{\Tilde{N}} K^{(l)}_{\mathtt{subgrad}} &\leq \left[b^2\left(\frac{bp^*}{c_{\tau_0}}\right)^2 \sum_{l = l_0 + 1}^{\Tilde{N}} \frac{1}{b^{2(\Tilde{N} - l)}}\right]\frac{D_0^2}{R^2}\frac{1}{\varepsilon^2} = O\left(\frac{1}{\varepsilon^2}\right). 
    \end{align*}

\subsection{Proof of Corollary~\ref{cor: smoothRateWith||-MRM}}
    Let $\Bar{\varepsilon} = \frac{c_{\tau_0}\rho \varepsilon}{bp^*}$, $\tilde{N} = \lceil\log_b(\frac{bp^*}{c_{\tau_0}\rho \varepsilon})\rceil$, and $\theta^{(l)} = \frac{\delta^{(l)}}{2\log(m+1)}$. By Theorem~\ref{thm: main||-MRMtheorem}, it suffices to show that $K^{(l_0)}_{\mathtt{smooth}}$ and $\sum_{l = l_0 + 1}^{\Tilde{N}} K^{(l)}_{\mathtt{smooth}}$ are bounded by $O(1 / \varepsilon).$ 
    Recall that $\mathtt{smooth}^{(l)}$ corresponds to Nesterov's accelerated method applied to the smoothed objective
    $$
    \Phi_{\tau^{(l)}_i, \theta^{(l)}}(y) \defeq \theta^{(l)} \log\left(\exp\left(\frac{(\tau^{(l)}_i f)^{\Gamma, e_0}}{\theta^{(l)}}\right) + \sum_{j=1}^m \exp\left(\frac{\varphi_{S_j}(y)}{\theta^{(l)}}\right)\right).
    $$

    First, we observe that all of the components $(\tau^{(l)}_i f)^{\Gamma, e_0}, \varphi_{S_1}, \dots, \varphi_{S_m}$ are all $L$-smooth and $M$-Lipschitz where $M = 1/R$ and $L = \max\{(1 + \frac{D_0}{R_0})^3\tau_0\beta, \frac{R + \beta D^2}{R^3}\}.$
    The smoothness and Lipschitz continuity of each identifier is verified below in Lemma~\ref{lem: HuberGaugeSmoothing}. The $(1 + \frac{D_0}{R_0})^3\tau_0\beta$-smoothness of $(\tau^{(l)}_i f)^{\Gamma, e_0}$ follows from~\cite[Corollary 1]{radial2} and noting $\tau^{(l)}_i\leq \tau_0$. The $1/R$-Lipschitz continuity of $(\tau^{(l)}_i f)^{\Gamma, e_0}$ follows from~\eqref{implication:RLipschitzOfConvexDual}. From these bounds, it follows that $\Phi_{\tau^{(l)}_i, \theta^{(l)}}$ is $\max\{(1 + \frac{R_0}{D_0})^3\tau_0\beta, \frac{R + \beta D^2}{R^3}\} + \frac{M^2}{\theta^{(l)}}$-smooth (see Appendix B of \cite{Beck2012}). Hence
    \begin{align*}
        K_{\mathtt{smooth}}&\leq 2\sqrt{\frac{2LD^2}{\delta^{(l)}/ \rho}+\frac{4M^2D^2\log(m+1)}{(\delta^{(l)}/ \rho)^2}}\\
        &\leq 2 \sqrt{\frac{2LD^2}{b} + 4M^2\log(m+1)D^2}\ \frac{\rho}{\delta^{(l)}}
    \end{align*}
    where the second inequality uses that all $l \geq l_0$ have $\frac{\delta^{(l)}}{\rho} < 1/b.$

    From this, it follows that $K_{\mathtt{smooth}}^{(l_0)} \leq  2 \sqrt{\frac{2LD^2}{b} + 4M^2\log(m+1)D^2}\ b^2$ is constant with respect to $\varepsilon$ as $1 / \delta^{(l_0)} \leq b^2/\rho.$ For $l > l_0,$ we have $1 / \delta^{(l)} = b^l \leq \frac{b}{\rho}\left(\frac{bp^*}{c_{\tau_0}}\right)\frac{1}{b^{\Tilde{N} - l}}\frac{1}{\varepsilon}.$ Therefore, the total iteration bound of Theorem~\ref{thm: main||-MRMtheorem} scales with $\varepsilon$ as
     \begin{align*}
     \sum_{l= l_0+1}^{\Tilde{N}}K_{\mathtt{smooth}}^{(l)} 
     &\leq 2 \sqrt{\frac{2LD^2}{b} + 4M^2\log(m+1)D^2} \left[ b\left(\frac{bp^*}{c_{\tau_0}}\right) \sum_{l= 0}^{\infty}1/b^l\right]\frac{1}{\varepsilon} = O(1/\varepsilon). 
     \end{align*}

\subsection{Proof of Corollary~\ref{cor: genGradRateWith||-MRM}}
    Note that $\gauge{S_1}{e_1}^2, \dots, \gauge{S_m}{e_m}^2$ are all $2\frac{R + \beta D^2}{R^3}$-smooth by \cite[Corollary 3.2]{liu2023gauges}. In addition, $(\tau^{(l)}_i f)^{\Gamma, e_0}$ is $(1 + \frac{R_0}{D_0})^3\tau^{(l)}_i\beta$-smooth by \cite[Corollary 1]{radial2}. Noting $\tau^{(l)}_i \leq \tau_0$, it follows that $(\tau^{(l)}_i f)^{\Gamma, e_0}, \gauge{S_1}{e_1}^2, \dots, \gauge{S_m}{e_m}^2$ are all $L=\max\{(1 + \frac{D_0}{R_0})^3\tau_0\beta, \frac{2(R + \beta D^2)}{R^3}\}$-smooth. Therefore, Theorem 2.3.5 of~\cite{nesterov-textbook} ensures that $K^{(l)}_{\mathtt{genGrad}} \leq 2\sqrt{LD_0^2} \cdot \sqrt{\frac{\rho}{\delta^{(l)}}}.$
    
    Since $1/ \delta^{(l_0)} \leq b^2/\rho,$ it follows that $K^{(l_0)}_{\mathtt{genGrad}} \leq 2\sqrt{LD_0^2}\ b$ is constant with respect to $\varepsilon$.
    Now, let $\Bar{\varepsilon} = \frac{c_{\tau_0}\rho \varepsilon}{bp^*}$ and $\Tilde{N} = \lceil\log_b(1/\Bar{\varepsilon})\rceil.$ For $l > l_0,$ we have $1 / \delta^{(l)} = b^l \leq \frac{b}{\rho}\left(\frac{bp^*}{c_{\tau_0}}\right)\frac{1}{b^{\Tilde{N} - l}}\frac{1}{\varepsilon}.$ Therefore, the total iteration bound of Theorem~\ref{thm: main||-MRMtheorem} scales with $\varepsilon$ as
    \begin{align*}
        \sum_{l = l_0+1}^{\Tilde{N}}K^{(l)}_{\mathtt{genGrad}} &\leq 2\sqrt{LD_0^2}\left[\sqrt{b}\left(\frac{bp^*}{c_{\tau0}}\right)^{\frac{1}{2}}\sum_{l = 0}^{\infty}(\sqrt{b})^{-l}\right]\frac{1}{\sqrt{\varepsilon}} = O(1 / \sqrt{\varepsilon}).
    \end{align*}

\section{Smoothness of the identifiers in Corollary~\ref{cor: smoothRateWith||-MRM}}

    \begin{lemma}\label{lem: HuberGaugeSmoothing}
    Let $S$ be convex and $e \in \interior S.$ Then $\varphi_{S} : \varspace \to \R,$ defined by
    $$
    \varphi_{S}(x) \defeq \begin{cases} \gauge{S}{e}(x) & \text{if } \gauge{S}{e}(x) > 1 \\
    \frac{1}{2}\gauge{S}{e}^2(x) + \frac{1}{2} & \text{otherwise}
    \end{cases},
    $$
    is convex and $1/R_e(S)$-Lipschitz. If $\frac{1}{2}\gauge{S}{e}^2$ is $L$-smooth, $\varphi_{S}$ is $L$-smooth. In particular, if $S$ is $\beta$-smooth, then $\varphi_{S}$ is $\frac{R_e(S) + \beta D_e(S)^2}{R_e(S)^3}$-smooth for $D_e(S) \defeq \sup\{\|x - e\| \mid x \in S\}$.
    \end{lemma}

    \begin{proof}
        We will show that $\varphi_S(x) = \frac{1}{2} + \inf_{y \in \varspace} \gauge{S}{e}(y) + \frac{1}{2}\gauge{S}{e}^2(e + x - y).$ Note that this will immediately impy $\varphi_S$ is convex and as smooth as $\frac{1}{2}\gauge{S}{e}^2$ because infimal convolutions preserve convexity and smoothness. Lipschitz continuity is also immediate as the subgradients of $\varphi_S$ and $\gauge{S}{e}$ are bounded above by $1/R_e(S).$

        It remains to show $\varphi_S(x) = \frac{1}{2} + \inf_{y \in \varspace} \gauge{S}{e}(y) + \frac{1}{2}\gauge{S}{e}^2(e + x - y).$ We prove that for any $\mu > 0,$
        \begin{equation}\label{eq: GaugeInfConvolution}
            \inf_{y \in \varspace} \gauge{S}{e}(y) + \frac{1}{2\mu}\gauge{S}{e}^2(e + x - y) = \begin{cases} \gauge{S}{e}(x) - \frac{\mu}{2} & \text{if } \gauge{S}{e}(x) > \mu \\
            \frac{1}{2\mu}\gauge{S}{e}^2(x) & \text{otherwise}.
            \end{cases}
        \end{equation}
        By the sum rule and chain rule, which apply as $\gauge{S}{e}$ and $\gauge{S}{e}^2$ are convex and real-valued, the necessary and sufficient condition to attain the infimum above is
        $$
        0 \in \partial \gauge{S}{e}(y) - \frac{1}{\mu} \gauge{S}{e}(e + x - y)\partial \gauge{S}{e}(e + x - y).
        $$ 
        The result in \eqref{eq: GaugeInfConvolution} holds because $y = e$ or $y = e + \left[1 - \frac{\mu}{\gauge{S}{e}(x)}\right](x - e)$ satisfies the sufficient condition accordingly as $\gauge{S}{e}(x) \leq \mu$ or $\gauge{S}{e}(x) > \mu.$ Indeed, since $0 \in \partial \gauge{S}{e}(e)$ and $\partial \gauge{S}{e}(x) \subset \partial \gauge{S}{e}(e)$ for any $x,$ it follows from convexity of the subdifferential that $\frac{1}{\mu}\gauge{S}{e}(x)\partial \gauge{S}{e}(x) \subset \partial \gauge{S}{e}(e)$ if $\gauge{S}{e}(x) \leq \mu.$ Therefore, $y = e$ attains the infimum if $\gauge{S}{e}(x) \leq \mu.$ 
        
        Now suppose $\gauge{S}{e}(x) > \mu$ and let $y = e + \left[1 - \frac{\mu}{\gauge{S}{e}(x)}\right](x - e).$ Recall that for any $\alpha > 0,$ we have $\gauge{S}{e}(e + \alpha (x - e)) = \alpha \gauge{S}{e}(x)$ and $\partial \gauge{S}{e}(e + \alpha(x - e)) = \partial \gauge{S}{e}(x).$ Noting that $e + x - y = e + \frac{\mu}{\gauge{S}{e}(x)}(x - e),$ we conclude that 
        $$
        \gauge{S}{e}(e + x - y) = \frac{\mu}{\gauge{S}{e}(x)}\gauge{S}{e}(x) = \mu \quad \text{and} \quad \partial \gauge{S}{e}(y) = \partial \gauge{S}{e}(x) = \partial \gauge{S}{e}(e + x - y).
        $$
        So, $\partial \gauge{S}{e}(y) = \frac{1}{\mu} \gauge{S}{e}(e + x - y)\partial \gauge{S}{e}(e + x - y),$ and $y$ satisfies the sufficient condition. 
    
        Plugging in the suitable minimizer $y = e$ or $y = e + \left[1 - \frac{\mu}{\gauge{S}{e}(x)}\right](x - e)$ accordingly into $\gauge{S}{e}(y) + \frac{1}{2\mu}\gauge{S}{e}^2(e + x - y)$ gives \eqref{eq: GaugeInfConvolution}. Finally, if $S$ is $\beta$-smooth and $D_e(S) < \infty,$ then $\frac{1}{2}\gauge{S}{e}^2$ is $\frac{R_e(S) + \beta D_e(S)^2}{R_e(S)^3}$-smooth by \cite[Corollary 3.2]{liu2023gauges}. Therefore, $\varphi_S$ is smooth with the same smoothness constant.
    \end{proof}

\end{document}